\documentclass[12pt,leqno]{article}
\usepackage[a4paper, margin=25.4mm]{geometry}
\usepackage{amssymb}
\usepackage{amsmath}
\usepackage{amsthm}
\usepackage{stmaryrd}
\usepackage{mathtools,mathrsfs}
\usepackage[pagebackref]{hyperref}
\hypersetup{citecolor=purple, linkcolor=blue, colorlinks=true}
\usepackage{booktabs}
\usepackage{enumerate}
\usepackage{tikz}
\usepackage{url}
\usepackage[T2A,T1]{fontenc}
\usepackage{xfrac}
\usepackage{authblk}  

\newtheorem{theorem}{Theorem}[section]
\newtheorem{proposition}[theorem]{Proposition}
\newtheorem{lemma}[theorem]{Lemma}
 
\newtheorem{hypothesis}[theorem]{Hypothesis}
\theoremstyle{definition}
\newtheorem{definition}[theorem]{Definition}

\newtheorem{remark}[theorem]{Remark}

\renewcommand{\le}{\leqslant}
\renewcommand{\ge}{\geqslant}
\renewcommand{\leq}{\leqslant}
\renewcommand{\geq}{\geqslant}
\newcommand{\lhdeq}{\trianglelefteqslant}    

\newcommand{\coloneq}{\vcentcolon=}      

\newcommand{\Aut}{\textrm{Aut}}

\definecolor{auburn}{rgb}{0.43, 0.21, 0.1}
\newcommand{\cC}{\mathscr{C}}
\newcommand{\cS}{\mathscr{S}}
\newcommand{\F}{{\mathbb F}}
\newcommand{\Z}{\mathrm{Z}}

\newcommand{\GL}{{\mathrm {GL}}}
\newcommand{\SL}{{\mathrm {SL}}}
\newcommand{\PSL}{{\mathrm {PSL}}}
\newcommand{\POm}{{\mathrm {P}\Omega}}
\newcommand{\PGL}{{\mathrm {PGL}}}
\newcommand{\PSU}{{\mathrm {PSU}}}
\newcommand{\PSp}{{\mathrm {PSp}}}

\newcommand{\Sp}{{\mathrm {Sp}}}

\newcommand{\SU}{{\mathrm {SU}}}
\newcommand{\SO}{{\mathrm {SO}}}
\newcommand{\Sz}{{\mathrm {Sz}}}
\newcommand{\Co}{{\mathrm {Co}}}
\newcommand{\HS}{{\mathrm {HS}}}
\newcommand{\Janko}{{\mathrm {J}}}
\newcommand{\Mathieu}{{\mathrm {M}}}
\newcommand{\Ru}{{\mathrm {Ru}}}

\newcommand{\Spin}{{\mathrm {Spin}}}
\newcommand{\GammaL}{\Gamma{\mathrm {L}}}

\newcommand{\nonsplit}[2]{#1\raisebox{0.6ex}{$\cdot$} #2}

\newcommand{\Out}{{\mathrm {Out}}}
\newcommand{\diag}{{\mathrm {diag}}}
\newcommand{\Lie}{{\mathrm {Lie}}}

\newcommand{\om}{\omega}

\author[1]{S. P. Glasby}
\affil[1]{\small Center for the Mathematics of Symmetry and Computation,
  \newline University of Western Australia, Perth 6009, Australia\newline
  \href{mailto:Stephen.Glasby@uwa.edu.au}{Stephen.Glasby@uwa.edu.au},
  \href{mailto:Cheryl.Praeger@uwa.edu.au}{Cheryl.Praeger@uwa.edu.au}}

\author[2]{Alice C. Niemeyer}
\affil[2]{\small Algebra and Representation Theory, RWTH Aachen University,
Pontdriesch 10-16, 52062 Aachen, Germany
  \href{mailto:alice.niemeyer@art.rwth-aachen.de}{alice.niemeyer@art.rwth-aachen.de}}

\author[1]{Cheryl E. Praeger}

\author[3]{A. E. Zalesski}
\affil[3]{\small Department of Mathematics, University of Brasilia, Brasilia - Federal District 70910-900, Brazil
  \href{mailto:alexandre.zalesski@gmail.com}{alexandre.zalesski@gmail.com}}

\title{Absolutely irreducible quasisimple linear groups containing elements of order a specified Zsigmondy~prime} 

\date{\today}

\begin{document}
\maketitle

\begin{abstract}
This paper is concerned with absolutely irreducible quasisimple subgroups $G$ of a finite general linear group $\GL_d(\F_q)$ for which some element $g\in G$ of prime order~$r$, in its action  on the natural module $V=(\F_q)^d$, is irreducible on a subspace of the form $V(1-g)$ of dimension $d/2$. 
We classify $G,d,r$, the characteristic $p$ of the field $\F_q$, and we identify those examples where the element $g$ has a fixed point subspace of dimension $d/2$. Our proof relies on representation theory, in particular, the multiplicities of eigenvalues of $g$, and builds on earlier results of DiMuro.
  \vskip2mm\noindent 
  {\bf Dedication: } To the memory of Richard A. Parker with gratitude for his friendship, and in recognition of his contribution to the modular character theory of quasisimple groups. 
\end{abstract}


\section{Introduction}

Absolutely irreducible quasisimple linear groups  arise naturally in Aschbacher's analysis~\cite{Asch} of subgroups of finite general linear groups, and finite  classical groups, as subgroups which preserve no natural geometrical structure on the natural module. Dealing with these groups is important both theoretically and computationally. State-of-the-art recognition algorithms for finite classical groups~\cite{DLLOB, GNP2023} seek a pair of elements of prime order which generate a smaller dimensional classical subgroup with high probability. Viewing such an element-pair as belonging to a smaller classical subgroup $X$ acting on the natural $d$-dimensional module for $X$, at least one element of the pair 
acts irreducibly on a subspace of dimension at least $d/2$.
This property is studied in \cite{GPPS} and later in \cite{DiMThesis}; and proving a probabilistic generation result for $X$ requires information about the absolutely irreducible quasisimple subgroups of $X$ that contain such an element. These elements are stingray elements (or ppd stingray elements) defined as follows.

\begin{definition}\label{d:ppd}
  Let $d, e,a $ be integers such that $1<e<d$ and $q=p^a$ where $p$
  is prime.  
  \begin{enumerate}
  \item[(a)] A  \emph{primitive prime divisor} of $q^e-1$ is a prime $r$
    dividing $q^e-1$ such that $r$ does not divide $q^i-1$ for any $i<e$;
    we sometimes call such a prime an \emph{$e$-ppd prime}.
  \item[(b)] For an $e$-ppd prime $r$, an element $g\in\GL_d(q)$ of order
    a multiple of $r$ is called an \emph{$e$-ppd element}. If such an
    element $g$ has order $r$, then $g$ is semisimple and preserves a
    decomposition of the natural module $V=\F_q^d$ of the form
    $V=\left(\oplus_{i=1}^t U_i\right)\oplus F$, where $t\geq 1$, $F$ is  the
    fixed point subspace of $g$, and $g$ leaves invariant and acts
    irreducibly on each $U_i$ with $\dim(U_i)=e$.
  \item[(c)] An element $g\in\GL_d(q)$ is called an \emph{$e$-stingray element} if $g$ preserves a
    decomposition $V=U\oplus F$, where $F$ is  the
    fixed point subspace of $g$ and $g$ acts
    irreducibly on $U$ with $\dim(U)=e$.  In particular $|g|\mid q^e-1$.
    An $e$-stingray element $g$ of order a multiple of an $e$-ppd prime $r$ is called
    an \emph{$e$-ppd stingray element} (and in most of this paper $|g|$ is equal to the $e$-ppd prime $r$). 
 \end{enumerate}   
\end{definition}

Each of the pair of elements referred to in the first paragraph above is an $e$-stingray element of $\GL_d(q)$, for some $e$, and for at least one element in the pair the parameter $e$ satisfies $d/2\leq e\leq d$. 
If $e>d/2$, then each $e$-ppd element of prime order $r$ is an $e$-stingray element, and the absolutely irreducible quasisimple subgroups $G$ of $\GL_d(q)$ containing an $e$-ppd element were classified in \cite{GPPS} (see \cite[Examples 2.6--2.9, especially Tables 2--8]{GPPS}); these results specify the parameter $e$ and the prime $r$ in terms of $e$.
\emph{Our focus in this paper is on $e$-ppd elements and $e$-ppd stingray elements with $e=d/2$.} 

The  absolutely irreducible quasisimple subgroups of $\GL_d(q)$ containing an $e$-ppd element with $e=d/2$ are classified  in DiMuro's thesis \cite{DiMThesis}. In fact, 
DiMuro's classification is given for all $e$ in the range $d/3<e\leq d/2$, and for  elements of prime power order~$r$, where $r$ satisfies the condition in Definition~\ref{d:ppd}(a). Information is given about the prime (power) $r$ in DiMuro's tables in \cite[Section 1.2]{DiMThesis} and the published paper \cite{DiM}, but the parameter $e$ is not identified explicitly. We note that, for $e=d/2$, either an $e$-ppd element is an $e$-stingray element, or it leaves invariant a decomposition of the natural module as a direct sum of two $e$-dimensional irreducible submodules, see Remark~\ref{r:e-ppd}.  \emph{Disentangling the relevant examples of groups containing an $e$-ppd element of prime order $r$ with  $e=d/2$, and deciding which of the groups contains an $e$-stingray element (not just an $e$-ppd element)  is not straightforward, and is the aim of this paper.} 

For these reasons we first summarise DiMuro's results for absolutely irreducible quasi\-simple subgroups $G$ of $\GL_d(q)$ containing an $e$-ppd element with $e=d/2$.  We focus on elements of order a ppd-prime because these are the elements required algorithmically. Since DiMuro's journal article~\cite{DiM} does not cover the case where $G$ is of Lie type in the same characteristic $p$, we have given an independent analysis of this case in Section~\ref{sec:samechar}, and in so doing we uncovered four groups missing from the tables in \cite{DiMThesis}, see Remark~\ref{r:missing} and Table~\ref{t:c9-samechar-tocheck}. 
The examples coming from the cases treated in DiMuro's journal article~\cite{DiM} are those in Tables~\ref{t:c9-permmod} -- \ref{t:c9-cross-i}.   
We were able to verify independently all the entries
in Tables~\ref{t:c9-alt}--\ref{t:c9-spo}, and small cases 
in Table~\ref{t:c9-cross-i}, with the aid of the AtlasRep package~\cite{AtlasRep} in {\sf GAP} and the lists of Hiss and Malle \cite{HMc}. We give some details about these checks in Remark~\ref{rem:verification}.
There may be gaps in DiMuro's arguments if $r$ is a proper prime power\footnote{According to DiMuro, see \cite[footnote 1 on page 2098]{BM}, `there is at present a gap in the argument concerning those elements whose orders are prime powers but not prime'.} though the results for prime order elements seem essentially sound, modulo our comments in Remark~\ref{rem:verification} and Remark~\ref{r:missing}.
Deciding which of these groups contains a stingray element required, in places, very delicate analysis of these groups, using methods developed by the fourth author in \cite{DiMPZ, DiMZ, Z06, Z08, Z23}.

The subgroups of $\GL_d(q)$ studied are \emph{quasisimple} groups $G$, that is $G/\Z(G)$ is a nonabelian simple group. As in \cite{DiMThesis,DiM, GPPS} we do not specify nondegenerate forms on the natural module  $V=(\F_q)^d$ which may be left invariant by a quasisimple group $G$, and we assume that  $G$ does not contain a classical group, where: 
\begin{equation}\label{d:class}
\begin{array}{l}
  \text{By a \emph{classical group} in $\GL_d(q)$ we mean a subgroup of the form $\SL_d(q)$,}\\
 \text{$\Sp_d(q)$ ($d$ even),  $\SU_d(q_0)$ ($q=q_0^2$), $\Omega^\pm_d(q)$ ($d$ even), or $\Omega^\circ_d(q)$ ($dq$ odd),}  \\
 \text{acting on its natural module $(\F_q)^d$ or the dual module.}
\end{array}    
\end{equation}
 Field automorphisms and diagonal automorphisms leave invariant the family of $(d/2)$-stingray elements in a classical group, but the situation is more complicated for graph automorphisms. We discuss this in  Remark~\ref{r:O8spin} for the $8$-dimensional orthogonal groups $\Omega_8^{+}(q)$.
 Since the subgroups of $\GL_2(q)$ are well known, we also assume that the dimension $d\geq4$.  
Our main result Theorem~\ref{t:stingray} firstly clarifies DiMuro's classification of such quasisimple groups which contain $(d/2)$-ppd elements, and
secondly identifies those   which contain a $(d/2)$-ppd stingray element of prime order.  It turns out that  $(d/2)$-ppd stingray elements in the groups in Theorem~\ref{t:stingray} are surprisingly rare, and for reference we collect the examples in Table~\ref{t:stelts}.

\begin{theorem} \label{t:stingray}
     Let $d$ be an even integer such that $d\geq 4$, and let $q=p^a$ for a prime $p$ and positive integer $a$. Suppose that $G<\GL_d(q)$ is quasisimple, is absolutely irreducible on the natural module $V=(\F_q)^d$, does not contain a classical subgroup as defined in \eqref{d:class}, and is not realisable over a proper subfield of $\F_q$. Then the following hold.
     \begin{enumerate}
         \item[(a)] {\rm\cite[Theorem~1.1]{DiM}}\ $G$ contains a ${(d/2)}$-ppd element of prime order $r$ if and only if $d, r, q$ and the group $G$ are as in one of the Tables~$\ref{t:c9-permmod}, \ref{t:c9-alt}, \ref{t:c9-spo}, \ref{t:c9-cross-u},  \ref{t:c9-cross-i}$, or $\ref{t:c9-samechar}$. In Tables~$\ref{t:c9-cross-i}$ and~$\ref{t:c9-samechar}$ we give the simple group $S=G/\Z(G)$ rather than $G$.

         \item[(b)] A group $G$ in one of these tables contains a ${(d/2)}$-ppd stingray element of prime order $r$ if and only if the last column of the relevant line contains $\checkmark$ {\rm(}see also Table~$\ref{t:stelts}${\rm)}.
     \end{enumerate}
\end{theorem}

\begin{table}
  \caption{Groups containing a $(d/2)$-ppd stingray element $g$ of order $r$ in Theorem~\ref{t:stingray}}
  \vskip 2mm
\begin{tabular}{llllll}
  \toprule
Table &$d$ &$r$ &  $G$ &  Comments\\
    \midrule
\ref{t:c9-permmod} &$\geq5$ & $d/2+1$ &  $A_{n}$, with $n=d-\delta$ & deleted permutation module \\
                 & &         &  & $q=p$, $\delta=1$ if $p\nmid n$, else $\delta=2$\\
\ref{t:c9-alt}   & $4$  &$3$ &  $\nonsplit{2}{A_5}, \nonsplit{2}{A_6}, \nonsplit{2}{A_7}$&  $p\geq5$\\
                &  $4$     & $3$    &$A_7$ &$q=2$, $g$ in Class 3B  \\
 \ref{t:c9-cross-u}   & $4$  &$3$ &  $\nonsplit{2}{\PSL_2(7)}, \nonsplit{2}{\PSU_4(2)}$ & $g$ in Class 3A, Class 3D resp.\\
                   & $8$  &$5$ &  $\nonsplit{2}{\POm_8^+(2)}$& $g$ in Class 5A\\
 \ref{t:c9-samechar}   & $4$  &$3$  & $\SL_2(q)$& $q\equiv 2\pmod{3}$ \\
                       & $4$  &$r\mid p^{a'}+1$  & $\PSL_2(q)$&  $q=p^{a'c}$, $c$ odd, $c>1$ \\
\bottomrule
\end{tabular}
\label{t:stelts}
\end{table}

\subsection{Some comments and an algorithmic application}\label{sub:alg}

The family of subgroups in Aschbacher's classification~\cite{Asch} corresponding to the subgroups $G$ in Theorem~\ref{t:stingray} is slightly larger than just the quasisimple ones. The family is often called $\cS$ or $\cC_9$,  and its members are  subgroups $H<\GL_d(q)$ such that the socle $S$ of $H$ modulo scalars is a nonabelian simple group. Moreover, $S$ has a quasisimple preimage $G$ in $H$ which is absolutely irreducible, and if $|H|$ is divisible by a $(d/2)$-ppd prime $r$, then $r$ divides $|G|$ and Theorem~\ref{t:stingray} applies, see Lemma~\ref{l:c9}.  We discuss this classification a little more in Subsection~\ref{s:alg2}.

The organisation of the subgroups of classical groups provided by Aschbacher's Theorem has been a useful means of attacking numerous problems concerning classical groups, especially questions about generation. We comment here about a particular question of this type. It relates to the main result of \cite[Theorem 2]{PSY} which provided a foundation for justifying new constructive recognition algorithms for classical groups in even characteristic which were effective where methods utilising involution centralisers were not available, see for example \cite{DLLOB}.

The result \cite[Theorem 2]{PSY} showed that, for a classical subgroup $X_d(q)\leq\GL_{d}(q)$ and stingray element $g\in X_d(q)$ as in Hypothesis~\ref{hyp:app} below,   there is a positive constant $c$ such that, with probability at least $c$, $g$ together with a uniformly distributed random conjugate $g^x$ generates $X_d(q)$ if $X_d(q)$ is not a symplectic group with $q$ even, or $\langle g,g^x\rangle=\Omega_d^\epsilon(q)$ if it is.
The proof involved estimating the probabilities that $\langle g,g^x\rangle$ lies in the various kinds of maximal subgroups of $G$, and these estimates were satisfactory for large dimensions $d$. However there was a problem  for small $d$ in getting a satisfactory estimate for the bound $p_9(X_d(q),g)$ on the probability  that $\langle g,g^s\rangle$ lies in a quasisimple $\cC_9$ subgroup (namely for the quantity `$p_9(X,d/2,q)$' defined in \cite[(4) on p.\,61]{PSY}, see also \cite[Section 13]{PSY}).  In fact no estimate was obtained when $d<18$,  see \cite[Lemma 13.1]{PSY}. Applying Theorem~\ref{t:stingray} allows us to provide a satisfactory estimate for $p_9(X_d(q),g)$ for smaller dimensions $d$, which  in all cases is an improvement of the estimates in \cite{PSY}. 

Assuming $d>8$ in Hypothesis~\ref{hyp:app} simplifies the proof of Theorem~\ref{t:c9}.

\begin{hypothesis}\label{hyp:app}
Let $Y=X_{d}(q)$ be a classical group  {\rm(}as in~\eqref{d:class}{\rm)} with $d$ even, $d>8$, and $X\in\{\SL,\SU,\Sp,\SO^\epsilon\}$, such that $Y$ contains a  $(d/2)$-ppd stingray element $g$ of ppd prime order $r$.  
\begin{itemize}
    \item Such elements exist in $Y$ when $X=\SL$ if and only if 
    $(d,q)\ne (12,2)$. For  such pairs $(d,q)$ they also exist in $Y$ if $X=\SU,\Sp,\SO^\epsilon$ and $d/2$ is odd, even, even, respectively (but not for the other parities of $d/2$).
\end{itemize}
We also consider $Y=X_{12}(2)$ with $X\in\{\SL,\Sp,\SO^\epsilon\}$, where $Y$ contains a  $6$-stingray element $g$ of order $9$.  For all these groups $Y$ and elements $g$, let
\begin{itemize}
    \item $p_9(Y,g)$ be an upper bound on the probability that $g$ and a uniformly distributed random conjugate $g^x$ ($x\in Y$) both lie in  a quasisimple $\cC_9$ subgroup $G$ of $Y$ such that $G$ is not realisable over a proper subfield, and is not contained in a proper classical subgroup of $Y$.
\end{itemize}

\end{hypothesis}

\noindent

\begin{theorem}\label{t:c9}
    Let $Y=X_d(q)$, $g, r, G$, and $p_9(Y,g)$ be as in Hypothesis~$\ref{hyp:app}$, and  
    suppose that $d>8$ and $p_9(Y,g)>0$. Then $d\equiv 0\pmod{4}$, $r=d/2+1$, $q=p$, and $Y=\Sp_d(2)$ (with $p=2$) or $\SO^\epsilon_d(p)$. Moreover,  $p_9(Y,g)\leq q^{-d^2/4+ 2d+3}$.    
\end{theorem}

We prove Theorem~\ref{t:c9} in Section~\ref{s:tc9}.
The results in this paper will be applied further to analyse new algorithms for recognition of finite classical matrix groups, which in particular will  provide a  useful improvement on the complexity of the state-of-the-art recognition algorithms for finite classical groups in \cite{DLLOB}, see \cite{GNP2022b,GNP2023}.
They may also be useful to answer questions about generation, for example,
generation of geometric maximal subgroups of classical groups, generalising~\cite{BEGHM} 

\section{Tables of results for Theorem~\ref{t:stingray}}

In this section we give tables of results for Theorem~\ref{t:stingray}. In all cases $G<\GL_d(q)$ is quasisimple and contains a $(d/2)$-ppd element $g$ of ppd prime order $r$. \emph{The condition $o_r(q)=d/2$ must hold in all cases}, and other conditions on $q=p^a$, or the prime $p$, in the tables are usually in addition to this. 

\begin{enumerate}
    \item The first two tables are for the cases where the simple group $G/\Z(G)$ is an alternating group. They relate to Propositions~\ref{p:delperm} and~\ref{p:alt2}, respectively, with Table~\ref{t:c9-permmod} for actions on the deleted permutation module for the natural action of $A_n$, and Table~\ref{t:c9-alt} for all other examples involving alternating groups.

    \item Table~\ref{t:c9-spo} gives the examples where $G/\Z(G)$ is a sporadic simple group, determined in Proposition~\ref{p:c9-spo}.

    \item Tables~\ref{t:c9-cross-u} and~\ref{t:c9-cross-i} give the examples where $G/\Z(G)$ is a Lie type simple group over a field of order $s$, $G/\Z(G)$ is not isomorphic to an alternating or symmetric group,  and the field order $s$ is not a power of $p$. They are determined in Proposition~\ref{p:cross}. Table~\ref{t:c9-cross-u} contains individual examples, while Table~\ref{t:c9-cross-i} gives infinite families of examples. 
    
    \item Table~\ref{t:c9-samechar}  gives the examples where $G/\Z(G)$ is a Lie type simple group in characteristic $p$, and is not isomorphic to an alternating or symmetric group. They are determined in Proposition~\ref{p:same}.  
    
\end{enumerate}

\begin{remark}\label{rem:verification}
We describe here the procedure we took to verify all entries
in Tables~{\ref{t:c9-alt}--\ref{t:c9-spo}}, and small cases
in Table~\ref{t:c9-cross-i}.
We used the lists by Hiss and Malle \cite{HMc} of  absolutely irreducible representations of degree up to $2500$ 
of finite quasi-simple groups, excluding groups of that are of Lie type in their defining characteristic.  
We extracted all groups $G$ from these lists that have  a representation in
an even dimension $d$ and whose
order is divisible by a prime $r$
for which $r\equiv 1 \pmod{d/2}.$
In addition, we performed the same
steps for the representations of quasi-simple groups related to $\PSL_2(q)$ for  prime powers $q $ with $q\le 1000$ as described in \cite[Table~2(b)-(d)]{HM}.
In cases where the characteristic of the representation is known,
we determined, with the help of the AtlasRep package~\cite{AtlasRep} in {\sf GAP}
\cite{GAP}, the smallest field afforded by  a given representation. \qed
\end{remark}

\clearpage

\begin{table}
\caption{ $G=S=A_n$ on the fully deleted permutation module over $\F_p$ (Proposition~\ref{p:delperm})}
\vskip 2mm
\begin{tabular}{lllccc}
  \toprule
Line &$r$ &$n$ & Cycle type of $g$ & $q=p$ & stingray? \\
    \midrule
1 &$d+1$ & $d+1$ & $n^1$ & $p\nmid n$& $\times$\\
2 &      & $d+2$ & $(n-1)^11^1$ & $p\mid n$& $\times$\\
3 &$d/2+1$ & $d+1$ & $r^11^{n-r}$ & $p\nmid n$& $\checkmark$\\
4 &     & $d+2$ & $r^11^{n-r}$ & $p\mid n$& $\checkmark$\\
5 &$d/2+1$ & $d+2$ & $r^2$ & $p=2$& $\times$\\
\bottomrule
\end{tabular}
\label{t:c9-permmod}
\end{table}

\begin{table}
\caption{The alternating (covering) groups not in Table~\ref{t:c9-permmod} (Proposition~\ref{p:alt2})}
\vskip 2mm
\begin{tabular}{llllllc}
  \toprule
Line &$G$ &$d$ & $r$&  $q=p^a$ &Comments& stingray? \\
    \midrule
1 &$\nonsplit{2}{A_5}, \nonsplit{2}{A_6}$ & 4 &  $3$ & $p\geq 5$& one class&$\checkmark$\\
  &        &  &  $5$ & $p\geq7$&&$\times$\\
2 &$\nonsplit{2}{A_7}$ & 4 & $3$ &$p\geq 5$&class $3B$&$\checkmark$\\
  &        &  & $5,7$ & $p\ne 2$ &&$\times$\\
3 &$A_7$ & 4 & $3$ &$q=2$ & class $3B$&$\checkmark$\\
4 &$\nonsplit{2}{A_7}$  & 6 & $7$ &
$q=9$
&&$\times$\\
5 &$\nonsplit{3}{A_7}$  & 6 & $7$ & $p\ne 3$ &&$\times$\\
6 &$\nonsplit{6}{A_7}$  & 6 & $7$ &$p\geq 5$ &&$\times$\\
7 &$A_9$ & 8 & $5$ &
$q=2$
&&$\times$\\
8 &$\nonsplit{2}{A_8}$ or $\nonsplit{2}{A_9}$ & 8 &  $5$ & $p\ne 2$&&$\times$\\
\bottomrule
\end{tabular}
\label{t:c9-alt}
\end{table}

\begin{table}[!ht]
\caption{Sporadic quasisimple  groups (Proposition~\ref{p:c9-spo})}
  \vskip 2mm
\begin{tabular}{clrrlc}
  \toprule
Line &$G$ &$d$ & $r$& $q=p^a$ & stingray?\\
    \midrule
1  &$\Mathieu_{11}$           & 10 & 11 & any $p$                                  & $\times$ \\
   &$\Mathieu_{12}, \nonsplit{2}{\Mathieu_{12}}$ & 10 & 11 & 
   $q=$
   for $\Mathieu_{12}$, $p\geq 3$  for $\nonsplit{2}{\Mathieu_{12}}$ & $\times$ \\
2  &$\nonsplit{3}{\Mathieu_{22}}$ & 6 & 7 & 
$q=4$& $\times$ \\
3  &$\nonsplit{2}{\Mathieu_{22}}$ & 10 & 11 &  
$p\geq 3$ & $\times$ \\
4  &$\Mathieu_{23}, \Mathieu_{24}$ & 22 & 23 &  $p\ne 2$ for $\Mathieu_{23}$, $q=3$
for $\Mathieu_{24}$ & $\times$ \\
5  &$\Janko_{1}$ & 20 & 11 &  
{$q=2$} & $\times$ \\
6  &$\Janko_{2}, \nonsplit{2}{\Janko_{2}}$ & 6 & 7 & {$q=4$}
for $\Janko_{2}$, $p\geq 3$ for $\nonsplit{2}{\Janko_{2}}$ & $\times$ \\
7  &$\Janko_{3}, \nonsplit{3}{\Janko_{3}}$ & 18 & 19 & {$q=9$}
for $\Janko_{3}$, $p\ne 3$ for $\nonsplit{3}{\Janko_{3}}$ & $\times$ \\
8  &$\Co_{1}, \nonsplit{2}{\Co_{1}}$ & 24 & 13 &  {$q= 2$} for $\Co_{1}$, $p\geq 3$ for $\nonsplit{2}{\Co_{1}}$ & $\times$ \\
9  &$\Co_{2}, \Co_{3}$ & 22 & 23 &  {$q=2$}
for $\Co_{2}$, {$q=2,3$}
for $\Co_{3}$ & $\times$ \\
10 &$\HS$ & 20 & 11 & {$q=2$}  
& $\times$ \\
11 &$\nonsplit{2}{\Ru}$ & 28 & 29 &  $p\geq 3$& $\times$ \\
12 &$\nonsplit{2}{\Sz}, \nonsplit{6}{\Sz}$ & 12 & 7 & {$q=3$} 
for $\nonsplit{2}{\Sz}$, $p\geq 5$ for $\nonsplit{6}{\Sz}$ & $\times$ \\
13 &$\nonsplit{3}{\Sz}, \nonsplit{6}{\Sz}$ & 12 & 13 &  {$q=4$}
for $\nonsplit{3}{\Sz}$, $p\geq 5$ for $\nonsplit{6}{\Sz}$ & $\times$ \\
\bottomrule
\end{tabular}
\label{t:c9-spo}
\end{table}

\begin{table}
\caption{Lie type quasisimple groups over $\F_s$ in cross-characteristic ($p\nmid s$): individual examples (Proposition~\ref{p:cross})}
  \vskip2mm
\begin{tabular}{clllll}
  \toprule
Line &$G$ &$d$ & $r$& {$q=p^a$}  
& stingray?\\
    \midrule
1& $\nonsplit{2}{\PSL_{2}(7)}$ & 4 & 3 & &$\checkmark$ class $3A$\\
     &             & 4 &  7 &  & $\times$ \\
2& $\nonsplit{4}{\PSL_{3}(4)}$ & 4 & 5 &
{$q=9$} &$\times$\\
3&  $\nonsplit{6}{\PSL_3(4)}$ &  6 & 7 & {$p\neq 2,3 $} & $\times$\\
4 &$\nonsplit{4}{\PSL_{3}(4)}$ & 8 & 5 & & $\times$\\
5& $\nonsplit{2}{\PSp_6(2)}$         & 8 & 5 & & $\times$\\
 & $\nonsplit{2}{{\rm P}\Omega_8^+(2)}$   & 8 & 5 & &$\checkmark$ class $5A$ \\
6& $\nonsplit{2}{G_{2}(4)}$ & 12 & 7, 13 & &$\times$, $\times$\\
7& $\nonsplit{2}{{}^2\kern-2pt B_2(8)}=\nonsplit{2}{\Sz(8)}$ & 8 & 13 & {$q=5$} 
&$\times$\\
8& $\nonsplit{2}{\PSU_{4}(2)}\cong \Sp_4(3)$ & 4 & 3 &$p\ge5$& $\checkmark$ class $3D$\\
 &                               & 4 & 5 && $\times$\\
9& $\nonsplit{3}{\PSU_{4}(3)}$ & 6 & 7 & {$q=4$}  
&$\times$\\
 & $\nonsplit{6}{\PSU_{4}(3)}$ & 6 & 7 & $p\ne 2,3$  &  $\times$ \\
 & $\PSU_3(3)$     & 6 & 7 &   &$\times$\\
\bottomrule
\end{tabular}
\label{t:c9-cross-u}
\end{table}

\begin{table}
\caption{Lie type quasisimple groups over $\F_s$ in cross-characteristic ($p\nmid s$, $q=p^a$): infinite families (Proposition~\ref{p:cross})}
  \vskip2mm
  \renewcommand{\arraystretch}{1.5} 
\begin{tabular}{cllllc}
  \toprule
Line &$S$ &$d$ & $r$& Conditions on $n,s$& stingray? \\
    \midrule
1& $\PSp_{2n}(s)$ & $\frac{s^n-1}{2}$ & $\frac{s^n+1}{2}$ & $n=2^c\geq2$, $s=u^{2^{c'}}$, & $\times$  \\
       & &  &     &$u$ an odd prime, $c'\geq0$ &\\
2& $\PSU_n(s)$ & $\frac{s^n-s}{s+1}$ & $\frac{s^n+1}{s+1}$& $n$ an odd prime & $\times$ \\
3& $\PSL_n(s)$ & $\frac{s^n-1}{s-1}-1$ & $\frac{s^n-1}{s-1}$& $n$ an odd prime & $\times$ \\
4& $\PSL_{2}(s)$ & $s$ & $s+1$ & $s=2^{2^c}\geq16$, $r$ a Fermat prime & $\times$\\
5& $\PSL_{2}(s)$ & $s-1$ & $\frac{s+1}{2}$ & $s=t^{2^{c}}$, $t$ an odd prime, $c\geq1$& $\times$\\
      &     &       &  $\frac{s+1}{2}$, $s$ & $s\geq7$, $s$ a prime& $\times$\\
6& $\PSL_{2}(s)$ &  $\frac{s-1}{2}$ &  $\frac{s+1}{2}$ & $s=t^{2^{c}}$, $t$ an odd prime, $c\geq1$& $\times$\\
      &     &       &  $\frac{s+1}{2}$, $s$ & $s\geq11$, $s$ a prime& $\times$\\
\bottomrule
\end{tabular}
\label{t:c9-cross-i}
\end{table}

\begin{table}
  \caption{Defining characteristic Lie type quasisimple groups (Proposition~\ref{p:same}). The primitive prime divisor $r$ of $q^{d/2}-1$ is not usually stated}
  \vskip2mm
\begin{tabular}{clclc}
  \toprule
Line &$S$ & $d$  & Comments& stingray? \\
    \midrule
1 &$\PSL_2(q^4)$   & $16$  & tw. tensor product subgroup \cite[Corollary 2.2]{S99}&$\times$\\
2 &$\PSL_4(q^2)$   & $16$  & tw. tensor product subgroup \cite[Table 1B, line 1]{S99} &$\times$\\
3 &$\PSU_m(q)$   & $20, 28$  & $(m,V)=(6,\Lambda^3 V_0)$, $(7,S^2V_0)$&$\times$\\
 &   & $28, 36$  & $(m,V)=(8,\Lambda^2 V_0)$, $(9,\Lambda^2V_0)$&$\times$\\
4 &$\PSU_m(q^{1/2})$   & $6,10$  & $(m,V)=(3, S^2 V_0)$, $(5,\Lambda^2V_0)$&$\times$\\
5 &$\PSp_4(q^2)$   & $16$ & tw. tensor product subgroup \cite[Table 1B, line 5]{S99}&$\times$\\
6 &$\PSL_2(q)$   & $4$  & in $\Sp_4(q)$, $q=p^a\geq7, p\geq 5$, $\exists$\ stingray elements iff & $\checkmark$\\
    &&&\hspace{1cm} $G=\SL_2(q)$, $q\equiv 2\pmod{3}$ and $r=3$, or&\\
    &&&\hspace{1cm}  $G=\PSL_2(q)$, $q=p^{a'c}$, $c$ odd, $c>1$, $r\mid p^{a'}+1$&\\
7 &$\PSL_m(q)$   & $6,10$  & $(m,V)=(3, S^2 V_0)$, $(5,\Lambda^2V_0)$&$\times$\\
8 &$\POm^\circ_m(q)$   & $8,16$  & $m=7,9$, $q$ odd, $V$ a spin module&$\times$\\
$9$ &$\POm^-_8(q^{1/2})$   & $8$  & $(1/2)$-spin representation in $\Omega_8^+(q)$& $\times$\\
10 &$\PSp_m(q)$   & $8, 16$  & $m=6,8$, $q$ even, $V$ a spin module&$\times$\\
11 &$\POm^+_{10}(q)$   & $16$  & $V$ a spin module &$\times$\\
12 &$G_2(q)$   & $6$ & $G_2(q)<\Sp_6(q)$, $q>p=2$&$\times$\\
13& ${}^3\kern-1pt D_4(q^{1/3})$ & $8$ &  ${}^3\kern-1pt D_4(q^{1/3})< \Omega_8^+(q)$&$\times$\\
\bottomrule
\end{tabular}
\label{t:c9-samechar}
\end{table}

\section{Stingray elements and subgroups of classical groups }

In this section we make some  observations about ppd-elements and stingray elements in classical groups, and we prove Theorem~\ref{t:c9}.

\subsection{Aschbacher's class \texorpdfstring{$\cC_9$}{} (or \texorpdfstring{$\cS$}{}) }\label{s:generallemma}

In \cite[Section 1]{Asch}, Aschbacher defined eight families  $\cC_1,\dots,\cC_8$  of proper subgroups of a finite classical group  on $V=\F_q^d$ according to various structures on $V$ such subgroups preserved, see also~\cite[Definition~2.1.3]{BHRD}. Collectively, these eight families are often referred to as \emph{geometric subgroups} of the classical group. For example, the family $\cC_1$ essentially consists of subspace stabilisers.  Aschbacher then proved~\cite[Theorem]{Asch} that each proper subgroup $H$ of a classical subgroup of $\GL_d(q)$ on $V=\F_q^d$, as defined in \eqref{d:class}, is either contained in a subgroup lying in some $\cC_i$ with $1\leq i\leq 8$, or  the generalised Fitting subgroup of $H/(H\cap Z)$, where $Z$ is the subgroup of scalars, is a nonabelian simple group $S$ and the covering group $G$ of $S$ in $H$  is absolutely irreducible  on $V$, and its representation on $V$ is not realisable over a proper subfield of $\F_q$. In particular, the subgroup $G$ is quasisimple and satisfies the conditions of Theorem~\ref{t:stingray}. 
The collection of subgroups $H$ with these properties has come to be known as $\cC_9$, or sometimes as $\cS$, for example,  in~\cite{BHRD}. Sometimes the definition of the family $\cS$ is more prescriptive, for example in \cite{BHRD} elements of $\cS$ are assumed to be maximal in the classical group, and are assumed not to lie in any of the families $\cC_i$, $i\leq 8$.

If $H < \GL_d(q)$ lies in the class $\cC_9$ for some classical group on $V$, then the quasisimple group $G$ described above is equal to  $H^{(\infty)}$, the last term in the derived series of $H$. Here we prove Lemma~\ref{l:c9}, mentioned in Subsection~\ref{sub:alg}, which shows that any $(d/2)$-ppd element lying in $H$ must in fact lie in $G$, and hence Theorem~\ref{t:stingray} may be applied.

\begin{lemma}\label{l:c9}
   Let $d, q=p^a, V$ be as in Theorem~$\ref{t:stingray}$ so $d$ is even and $d\geq4$. Let  $Z$ be the group of scalars in $\GL_d(q)$, and suppose that $q^{d/2}-1$ has a primitive prime divisor $r$. Suppose also that $H<\GL_d(q)$ lies in the family $\cC_9$ described above, so that $G \coloneq  H^{(\infty)}$ is quasisimple and satisfies the hypotheses of Theorem~$\ref{t:stingray}$ and $S\coloneq G/(G\cap Z)$ is a nonabelian simple group. Then the $r$-parts $|H|_r=|G|_r=|S|_r$ are equal.  
\end{lemma}

\begin{proof}
  Since $H$ is absolutely irreducible, it follows that $S\leq H/(H\cap Z)\leq \Aut(S)$, so $|H|=|H\cap Z|\cdot|S|\cdot x$, where $x$ divides $|\Out(S)|$; also
 $|G|=y\cdot|S|$, where $y$ divides $|H\cap Z|$. There is nothing to prove if $r$ does not divide $|H|$, so suppose that $r$ divides $|H|$.  
  Since $o_r(q)=d/2\geq 2$,
  the prime $r$ does not divide $|Z|=q-1$, and hence $|G|_r=|S|_r$ and $r$ divides $|S|\cdot x$. Also $r\geq d/2+1\geq3$, so $r$ is odd.   To prove the lemma we assume to the contrary that $r$ divides $|\Out(S)|$ and derive a contradiction. First, $S$ is not isomorphic to an alternating group or a sporadic simple group since in these cases $\Out(S)$ is a $2$-group. Thus $S$ is a Lie type simple group over a field of order $s$. If $s$ is coprime to $p$, then recalling that $r\geq d/2+1$, it follows from the results of Landazuri and Seitz~\cite{LS}, that $d/2+1$ is strictly greater than  the odd part $|\Out(S)|_{2'}$ (see Table~5.3.A in \cite{KL}). Thus $s=p^b$ for some $b$, and $r$ divides the order of a field automorphism of $S$. Arguing as in \cite[last paragraph of Section 7]{GPPS}, since the module $V$ is invariant under this field automorphism, it follows from \cite[Proposition 5.4.2(i)]{KL} and Steinberg's twisted tensor product theorem \cite[Theorem 5.4.5]{KL}  that $d\geq 2^r$. Hence $r-1\geq d/2\geq 2^{r-1}$, which is a contradiction since $r$ is an odd prime. Thus $r$ is coprime to $|\Out(S)|$ and $|H|_r=|G|_r=|S|_r$, as claimed.
\end{proof}

\subsection{An extension of Theorem~\ref{t:stingray}}\label{s:alg2}

As preparation for the proof of Theorem~\ref{t:c9}, we give in Proposition~\ref{p:122} an equivalent of Theorem~\ref{t:stingray} for the case where $(d,q)=(12,2)$ and $6$-stingray elements of order $9$. We note that, while $2^6-1$ does not have a primitive prime divisor, the integer $9$  is a \emph{primitive prime-power divisor} of $2^6-1$ in the sense that it does not divide $2^i-1$ for any $i<6$.

In the proof of Proposition~\ref{p:122}, 
we use the definitions of the Aschbacher classes $\cC_1,\dots,\cC_8$ given  in~\cite[Definition~2.1.3]{BHRD}; and we recall that the class $\cS$ given in~\cite{BHRD} comprises groups which are maximal in the classical group, and are not subgroups of any group in $\cup_{i\leq 8}\cC_i$. 
Recall also that a subgroup $H$ of $\GL_{12}(2)$ is a $\cC_9$-subgroup precisely when  $H^{(\infty)}$ is quasisimple and  absolutely irreducible on the natural module $\F_{2}^{12}$, and $H$ does not contain a classical subgroup of $\GL_{12}(2)$ as defined in \eqref{d:class}. (The condition that the representation is not realisable over a proper subfield is vacuously true for the field $\F_2$.) 
In particular, since $H^{(\infty)}$ is absolutely irreducible, $H$ is not contained in a subgroup of $\cC_3$; and since $d=12$ is not a prime power or a proper power, $H$ is not contained in a subgroup of $\cC_i$ for $i=6,7$. Thus if a $\cC_9$-subgroup $H$ of $\GL_{12}(2)$ is contained in a subgroup in $\cC_i$ for some $i\leq 8$, then $i=2$ or $i=4$.

Note that the $\cC_9$-subgroup $H$ may not be maximal in $\GL_{12}(2)$, or
in any classical subgroup of $\GL_{12}(2)$. Also, since $\GL_{12}(2)$
contains no scalar matrices other than $I_{12}$ any absolutely irreducible
subgroup of $\GL_{12}(2)$ has a trivial center. Thus the absolutely
irreducible quasisimple group $H^{(\infty)}$ must be a simple group $S$, and the
possibilities for $S$ are severely constrained by number theory
($|S|$ divides $|\GL_{12}(2)|$) and representation~theory.

\begin{proposition}\label{p:122}
Suppose that $H$ is a $\cC_9$-subgroup of $\GL_{12}(2)$, that is, $H^{(\infty)}$ is quasisimple and is absolutely irreducible on $V=\F_2^{12}$, and $H$ does not contain a classical subgroup of $\GL_{12}(2)$ as defined in $\eqref{d:class}$.
Then 
\[
H^{(\infty)}\in\{
\PSL_2(25), A_{14}, A_{13}, \PSL_3(3)\}.
\]
However, none of these possibilities for $H$ contains a
$6$-stingray element of order $9$.
\end{proposition}

\begin{proof}
Recall our remarks preceding Proposition~\ref{p:122}. If $H\in\cS$, that is, if  $H$ is a maximal $\cC_9$-subgroup not contained in any subgroup in $\cup_{i\leq 8}\cC_i$, then by \cite[Tables 8.77--8.85]{BHRD},  $H^{(\infty)}\in\{ 
\PSL_2(25), A_{13}, A_{14}, \PSL_3(3)\}$. The example $A_{14}$ (which acts on the deleted permutation module) contains the other possibility $A_{13}$. 
Any additional $\cC_9$-subgroup $H$ must be contained in a maximal subgroup $M\in \cup_{i\leq 8}\cC_i$, and as discussed above, $M\in \cC_2\cup\cC_4$, and these subgroups $M$ are 
\[
\GL_6(2)\wr S_2, \GL_4(2)\wr S_3, \GL_3(2)\wr S_4, \GL_2(2)\wr S_6,  
\GL_6(2)\otimes\GL_2(2),\GL_4(2)\otimes\GL_3(2).
\]
Thus $H^{(\infty)}\le M$ is a simple irreducible subgroup.
If $H\le\GL_a(2)\wr S_b$ where $12=ab$ and $b=2,3,4$, then $S_b$ is solvable so
$H^{(\infty)}\le\GL_a(2)^b$ acts reducibly, a contradiction.
If $b=6$, i.e. $H\le \GL_2(2)\wr S_6$, then $H^{(\infty)}$ equals $A_5$ or $A_6$.
However, $A_5$ and $A_6$ act reducibly, see \cite[pp.\;2, 4]{Modat}. This eliminates the $\cC_2$
possibilities for $M$.
If $H\le\GL_a(2)\otimes\GL_b(2)$ where $b=2,3$, then as
$H^{(\infty)}$ is simple and $\GL_b(2)$ is soluble, $H^{(\infty)}$ lies in
$\GL_a(2)\otimes1$, so $H$ is reducible, a
contradiction. Thus $M\not\in\cC_4$, so
$H^{(\infty)}\in\{\PSL_2(25), A_{13}, A_{14}, \PSL_3(3)\}$. A search with
{\sc Magma}~\cite{Magma} shows that $\PSL_2(25), A_{13}, A_{14}, \PSL_3(3)$
have no absolutely irreducible simple subgroups. Hence
$H^{(\infty)}\in\{\PSL_2(25), A_{13}, A_{14}, \PSL_3(3)\}$ are the only
possibilities.

Suppose now that $H$ contains a $6$-stingray element $g$ of order $|g|=9$.
Then since $H\le N_{\GL_{12}(2)}(H^{(\infty)})$, and since 
$|N_{\GL_{12}(2)}(H^{(\infty)})/H^{(\infty)}|$ divides $|\Aut(H^{(\infty)})|$
which divides $4$, it follows that $g$ lies in $H^{(\infty)}$.
However, $H^{(\infty)}\ne \PSL_2(25)$ or
$\PSL_3(3)$ since in these cases $H^{(\infty)}$ contains no elements of order $9$.
Thus $H^{(\infty)}=A_{13}$ or $A_{14}$ acting on the deleted permutation module.
In these cases $H^{(\infty)}$ contains elements of order $9$; these
are either $9$-cycles, or the product of a $9$-cycle and a $3$-cycle.
By \cite[Lemmas 4.3 and 4.4]{PSY} it follows that the fixed point subspace
of $g$ in $\F_{2}^{12}$ has dimension at most $4$, and hence $g$ is not
a $6$-stingray element. This completes the proof.
\end{proof}

The following proof of Theorem~\ref{t:c9} relies on Theorem~\ref{t:stingray}.
  However, our upcoming proof of Theorem~\ref{t:stingray} does not use Theorem~\ref{t:c9}.

\subsection{Proof of Theorem~\ref{t:c9}}\label{s:tc9}

Let $Y=X_d(q)$, $g$, and $p_9(Y,g)$ be as in Hypothesis~\ref{hyp:app} and assume that $d>8$ and $p_9(Y,g)>0$. It follows from Proposition~\ref{p:122} that $(d,q)\ne (12,2)$. Hence $g$ is a $(d/2)$-ppd stingray element of ppd prime order $r$, and so we may apply Theorem~\ref{t:stingray}. The groups to be treated are therefore listed in Table~\ref{t:stelts}, and since $d>8$, the only possible  quasisimple group is $G=A_{n}$ and the only prime $r$ is $d/2+1$, where $n=d+\delta$, $\delta=1$ if $p\nmid n$, and otherwise $\delta=2$. We note that, since $r=d/2+1$ is odd, we must have $d\equiv 0\pmod{4}$.  The action of $G=A_{n}$ is on the deleted permutation module, which we discuss further in Section~\ref{ss:perm}, and we determine here the contribution these groups make to $p_9(Y,g)$. 
By \cite[pp.\,186--7]{KL} (see also Proposition~\ref{p:delperm}), the group $A_{d+\delta}$ is embedded in either an orthogonal group $\POm_d^\epsilon(p)$ for some $\epsilon=\pm$, or a symplectic group $\Sp_d(p)$ (if $p=2$ and $d\equiv 0\pmod{4})$. 
Since by assumption $G$ is a $\cC_9$-group, the representation of $A_{d+\delta}$ cannot be realised over a proper subfield of $\F_q$, and hence $q=p$. Also by assumption $G$ is not contained in a proper classical subgroup of $Y$, and so $X$ must be either $\Sp_d(2)$ with $p=2$, or $\SO^\epsilon_d(p)$. Then by \cite[Lemma 4.6]{PSY}, $p_9(Y,g)\leq q^{-d^2/4+ 2d+3}$.

\section{Comments on the proof strategy for Theorem~\ref{t:stingray}}\label{s:strategy}

We consider quasisimple subgroups $G<\GL_d(q)$ with $d$ even, $d\geq 4$, and $q=p^a$ for a prime $p$ and positive integer $a$. We assume that $|G|$ is divisible by  a primitive prime divisor $r$ of $q^{d/2}-1$, that is to say, $r$ is a $(d/2)$-ppd prime and $G$ contains a $(d/2)$-ppd element of order $r$. %
Now $r=r_0(d/2)+1$ for some integer $r_0$  (see \cite[Remark 1.1]{GPPS}), so in particular $r\geq d/2 +1\geq3$, which means that $r$ is an odd prime and if $r=3$ then $d=4$.   Also by the definition of a ppd element, $q$ has order $d/2$ modulo $r$, that is $o_r(q)=d/2$. 
The quotient $S\coloneq  G/\Z(G)$ is a nonabelian simple group, and since
$o_r(q)=d/2$, the prime $r$ does not divide $|\Z(G)|$, and hence $r$ divides $|S|$. 
To prove Theorem~\ref{t:stingray}(a) we use DiMuro's paper \cite{DiM} as the main resource for the cases where $S$ is not of Lie type in characteristic $p$, in Sections~\ref{s:alt}, \ref{s:spor}, and \ref{s:cross}. In the remaining `same characteristic' case we give some proof details in Section~\ref{sec:samechar} (and identify the relevant tables and sections of \cite{DiMThesis} where a proof can be found).     

Our main objective is to find all quasisimple subgroups $G<\GL_d(q)$ containing a $(d/2)$-ppd stingray element of prime order $r$. Note that each $(d/2)$-stingray element is semisimple, as the prime order $r$ is different from the characteristic. Also, these stingray elements act on the  natural module $V=\F_q^d$ as  `almost cyclic' matrices which have been studied extensively by the fourth author and colleagues, and are defined as follows.

\begin{definition}
{\rm
A matrix $g\in\GL_d(q)$ is \emph{almost cyclic} if, for some $k\in [0,d]$, some scalar multiple $zg$ ($z\in\F_q^\#$) is conjugate in $\GL_d(q)$ to a block diagonal matrix ${\rm diag}(I_{d-k}, g_1)$ where $g_1\in\GL_k(q)$ is a cyclic matrix (that is, the minimal and characteristic polynomials of $g_1$ are the same).  In particular, each $(d/2)$-stingray element is almost cyclic (with $k=d/2$ and $g_1$ irreducible), but the converse is not true. 
}
\end{definition}
 The papers \cite{DiMPZ} and \cite{DiMZ, Z08} contain classifications of various classes of absolutely irreducible quasisimple groups containing almost cyclic elements, and our proofs rely heavily on this work. 
We discuss in Remark~\ref{r:e-ppd}, the various kinds of  $(d/2)$-ppd elements and the links with almost cyclic elements.

\begin{remark}\label{r:e-ppd}
 Suppose that $d$ is even, $d\geq4$, and that $g\in\GL_{d}(q)$ is a $(d/2)$-ppd element of prime order, that is,  $r\coloneq |g|$ is a primitive prime divisor of $q^{d/2}-1$. Then $g$ satisfies precisely one of the following:
\begin{enumerate}
    \item[(1)] The element $g$ is a $(d/2)$-stingray element; that is it has a $(d/2)$-dimensional fixed point space and is irreducible on a unique $(d/2)$-dimensional complementary subspace.
\item[(2)] The element $g$ is semisimple with characteristic polynomial $f_1(t)f_2(t)$ where each $f_i$ is irreducible of degree $d/2$, and either 
\begin{enumerate}
    \item[(2.i)] $f_1\ne f_2$ which implies that $g$ is regular semisimple, and in fact $g$ is cyclic, or
    \item[(2.ii)] $f_1=f_2$, so $g$ is not regular semisimple, and is not cyclic. 
\end{enumerate} 
\end{enumerate}
Elements of types (1) and (2.i) are almost cyclic, while those of type (2.ii) are not. In particular, if $|g|=3$, then $q\equiv 2\pmod{3}$, $d=4$, and there is a unique irreducible polynomial of degree $2$, so either (1) or (2.ii) holds for $g$.\qed  
\end{remark}

To  determine whether a given element $g$ of a quasisimple subgroup $G\leq\GL_d(q)$, with $r\coloneq |g|$ a $(d/2)$-ppd prime, is a $(d/2)$-stingray element, it is often helpful to  inspect the Brauer characters in \cite{Modat} for those representations where $p$ divides $|G|$, or the character tables in \cite{Atlas} for those representations  where $p$ does not divide $|G|$. If $r=1+k(d/2)$ then a $(d/2)$-stingray element $g$ is conjugate in $\GL_d(\overline{\F_q})$ to a diagonal matrix ${\rm diag}(\eta_1, \eta_2,\dots, \eta_{d/2}, 1,\dots,1)$, where the $\eta_i$ are pairwise distinct primitive $r^{\rm th}$ roots of unity in $\overline{\F_q}$, an algebraic closure of $\F_q$. The Brauer character is essentially the trace of this matrix, namely $\chi(g)=d/2 + \sum_{i=1}^{d/2} \eta_i$, with the $r^{\rm th}$ roots of unity replaced by complex $r^{\rm th}$ roots of unity. 
We use the following fact:

 \begin{lemma}\label{l:rootsC}
Let $r$ be an odd prime, and let $\eta_1,\dots, \eta_{r-1}$ be the nontrivial $r^{\rm th}$ roots of unity over $\mathbb{C}$. If $\sum_{i=1}^{r-1} a_i\eta_i=c$, with $a_1,\dots, a_{r-1}, c\in  \mathbb{Q}$, then 
$a_1=\dots =a_{r-1}=-c$.
 \end{lemma}  

 \begin{proof}
Since $r$ is an odd prime, the subfield $\mathbb{Q}(\eta)$ generated by a non-identity $r$th root of unity $\eta$ contains $\mathbb{Q}$ with index $r-1$, the nontrivial $r$th roots of unity are precisely $\eta,\dots,\eta^{r-1}$, and they are linearly independent over $\mathbb{Q}$. 
Moreover, since $\eta^r=1$ and $\eta\ne 1$, we have $\sum_{i=0}^{r-1} \eta^i= (\eta^r-1)/(\eta-1)=0$, or equivalently, $\sum_{i=1}^{r-1} \eta^i= -1$.
Suppose that $\sum_{i=1}^{r-1} a_i\eta^i=c$, with $a_0,\dots, a_{r-1}, c\in\mathbb{Q}$. Then adding this equation to the equation $c\sum_{i=1}^{r-1} \eta^i= -c$, we have 
$\sum_{i=1}^{r-1} (a_i+c)\eta^i=0$, and hence $a_1=\dots =a_{r-1}=-c$.   
 \end{proof}

We illustrate how this fact may be used in Sections~\ref{s:alt} and~\ref{s:cross}.

\begin{lemma}\label{l:don2}
 Let $g\in G$, where $G$ is a quasisimple subgroup of  $G\leq\GL_d(q)$, with $r\coloneq |g|$ a $(d/2)$-ppd prime. Let $\chi$ be the Brauer character of $G$ {\rm(}see~{\rm\cite[\S4]{Modat}}{\rm)}.
 \begin{enumerate}
     \item[(a)] If $r=1+d/2$, then $g$  is a $(d/2)$-stingray element if and only if $\chi(g)=d/2-1$.
    \item[(b)] If $r=1+d$, then $g$  is of type $(2.i)$ if $\chi(g)=-1$, of type $(2.ii)$ if $\chi(g)=2N$, and a $(d/2)$-stingray element if  $\chi(g)=d/2+N$, where $N$ is a sum of $d/2$ pairwise distinct nontrivial $r$th roots of unity.
 \end{enumerate}
    
\end{lemma}

\begin{proof}
As discussed before Lemma~\ref{l:rootsC}, a $(d/2)$-stingray element $g$ is conjugate in $\GL_d(\overline{\F_q})$ to a diagonal matrix ${\rm diag}(\eta_1, \eta_2,\dots, \eta_{d/2}, 1,\dots,1)$, where the $\eta_i$ are pairwise distinct primitive $r^{\rm th}$ roots of unity in $\overline{\F_q}$, so $\chi(g)=d/2+\sum_{i=1}^{d/2}\eta_i$. 

    (a) If $r=d/2+1$, then the $\eta_i$ in the previous paragraph constitute all of the nontrivial $r^{\rm th}$ roots, and so $\sum_{i=1}^{d/2}\eta_i=-1$, and $f(t)=\prod_{i=1}^{d/2}(t-\eta_i)=(t^r-1)/(t-1)$ is the unique monic irreducible polynomial of degree $d/2$. Thus if $g$ is a stingray element then $\chi(g)=d/2-1$. On the other hand, if $g$ is not a stingray element, then by Remark~\ref{r:e-ppd}, $g$ is of type (2.ii) and so  $\chi(g)=2\sum_{i=1}^{d/2}\eta_i=-2$.

    (b) If $r=d+1$, then the $\eta_i$ above constitute half of the nontrivial $r^{\rm th}$ roots, and $f(t)=\prod_{i=1}^{d/2}(t-\eta_i)$ is one of the precisely two  monic irreducible polynomials over $\F_q$ of degree $d/2$. Let $N\coloneq \sum_{i=1}^{d/2}\eta_i$. If $g$  is of type $(2.i)$ then the characteristic polynomial of $g$ equals $(t^r-1)/(t-1)$ and $\chi(g)=-1$, the sum of all the nontrivial $r^{\rm th}$ roots of unity. If $g$ is of type $(2.ii)$ with characteristic polynomial $f(t)^2$ then $\chi(g)=2N$, and if $g$ is a $(d/2)$-stingray element then $\chi(g)=d/2+N$.
\end{proof}

Finally, we make some comments about the definition of classical groups in~\eqref{d:class}  in the cases of unitary groups (Remark~\ref{r:U}) and $8$-dimensional orthogonal groups (Remark~\ref{r:O8spin}). In the latter case additional conjugacy classes of orthogonal subgroups in $\GL_8(V)$ are also referred to as lying in the $\cC_8$-class of geometric subgroups (see Section~\ref{s:generallemma}). In particular we discuss the existence of $4$-stingray elements in the $8$-dimensional representations of the spin groups $H=\Spin_8^{\epsilon}(q)$, where $\epsilon=\pm$. 

\begin{remark}\label{r:U}
    Here we note that, if the dimension of an irreducible representation of $\SU_m(q)$ over $\F_{q^2}$ equals $m$, then 
the representation is the natural one up to quasi-equivalence, that is, up to twisting by a field automorphism. To see this, it follows from \cite[Theorem 43]{St} that every absolutely irreducible representation of $\SU_m(q)$ extends to a representation of $\SL_m(F)$, where $F$ is an algebraically closed field containing $\F_{q}$. The so-called $p$-restricted representations of degree $m$ are the natural one and its dual (see, for example, \cite[Table 2]{Lu}), and  all others of degree $m$ are obtained from these by applying the Frobenius map arising from the map $x\to x^{p^k}$
for $x\in F$. (The latter follows from the Steinberg tensor product theorem, which implies that the number of terms in the product equals 1 as there is no irreducible representation of $\SL_m(F)$ of dimension strictly between $1$ and $m$, \cite[Table 2]{Lu}.) Further, the group in the image of $\SU_m(q)$ under a representation is not affected by a twisting automorphism, and it follows that each image of $\SU_m(q)$ in $\GL_m(q^2)$ is conjugate to the stabiliser of a Hermitian form on $\F_{q^2}^m$ (see \cite[Lemma 1.8.6]{BHRD}), and hence is a classical subgroup of $\GL_m(q^2)$ as in \eqref{d:class}.\qed
\end{remark}

\begin{remark}\label{r:O8spin}    
  For $\epsilon=\pm$, the spin group $H=\Spin^{\epsilon}_n(q)$ is a certain
  central extension of $\Omega^{\epsilon}_n(q)$. If $q$ is odd and $n=2m$ is even, then
  $\Spin^{\epsilon}_n(q)\cong\nonsplit{2}{\Omega^{\epsilon}_n(q)}$ has the following structure~\cite[\S\,3.9.3]{W}
  \[
  \Spin^{\epsilon}_n(q)\cong
  \begin{cases}
    \nonsplit{2}{\POm^{\epsilon}_n(q)}&\textup{if $q^m\equiv -\epsilon\pmod{4}$,
      and $\Omega^{\epsilon}_n(q)\cong\POm^{\epsilon}_n(q)$,}\\
    \nonsplit{4}{\POm^{\epsilon}_n(q)}&\textup{if $q^m\equiv\epsilon\pmod{4}$, $m$ odd,
      and $\Omega^{\epsilon}_n(q)\cong\nonsplit{2}{\POm^{\epsilon}_n(q)}$,}\\
    \nonsplit{2^2}{\POm^{\epsilon}_n(q)}&\textup{if $q^m\equiv\epsilon\pmod{4}$, $m$ even,
      and $\Omega^{\epsilon}_n(q)\cong\nonsplit{2}{\POm^{\epsilon}_n(q)}$.}\\
  \end{cases}
  \]
  These groups have irreducible representations in characteristic $p$ of dimension $2^{m-1}$, the so-called spin representations, and in particular $2^{m-1}=2m=n$ when $n=8$ yielding additional $8$-dimensional representations of orthogonal groups in this special case.
  
  Suppose now that $n=2m=8$ and $q$ is odd, and let $r$ be a primitive prime divisor of $q^4-1$ (note that such a prime $r$ exists). Then  $q^m = q^4 \equiv 
 1\pmod{4}$, so $H\coloneq\Spin^{\epsilon}_8(q)$ is either $\nonsplit{2}{\POm^{-}_8(q)}$ with $\epsilon = -$ and $\Omega^{\epsilon}_8(q)\cong\POm^{\epsilon}_8(q)$; or $\nonsplit{2^2}{\POm^{+}_8(q)}$ with $\epsilon = +$ and  $\Omega^{\epsilon}_8(q)\cong\nonsplit{2}{\POm^{\epsilon}_8(q)}$.
 Let $V=\mathbb{F}_q^{8}$, and consider first the natural
  representation $\nu:H\rightarrow \GL(V)$ with $\nu(H)=\Omega_8^{\epsilon}(q)$ acting naturally on $V$, that is, $\nu(H)$ leaves invariant a nondegenerate quadratic form of type $\epsilon$ on $V$. For each $\epsilon=\pm$, $\nu$ has kernel of order $2$. Moreover, $\nu(H)$
  contains a subgroup $\Omega_4^-(q)\times \Omega_4^{-\epsilon}(q)$
  preserving an orthogonal decomposition $V=U\oplus U^\perp$ where
  $U$ is a nondegenerate $4$-space of minus type, and this subgroup
  contains a $4$-stingray element of order $r$ acting irreducibly on
  $U$ and fixing $U^\perp$ pointwise.

  We consider now the other representations of $H=\Spin^{\epsilon}_8(q)$. If $\epsilon=-$, then this representation $\phi:H\to \GL_8(q^2)$ has $\phi(H)\cong H= \nonsplit{2}{\POm^{-}_8(q)}$ and is not realisable over $\mathbb{F}_q$, as noted in the proof of \cite[Proposition~5.4.9(iii)]{KL}. Thus we only have the natural  conjugacy classes of orthogonal subgroups of $\GL_8(q)$ of minus type, that is to say, the groups preserving a nondegenerate quadratic form of minus type on the natural module $\F_q^8$ for $\GL_8(q)$. 
  
  The situation for $H=\Spin_{8}^{\epsilon}(q)$ with $q$ odd and $\epsilon=+$ is more interesting. Here the centre  $\Z(H)$ is an elementary abelian group of order four, and there are three inequivalent irreducible  representations  of $H=\Spin_{8}^{+}(q)$ of degree $8$. For each of these representations $\phi_i:H\to\GL_8(q)$ (where $i=1,2,3$), the image $\phi_i(\Z(H))$ has order two, and the three  representations have 
  pairwise distinct central kernels of order $2$. The group $H$ has an outer
  automorphism of order $3$ (a triality automorphism), permuting cyclically the three central involutions of $\Z(H)$, and the three representations $\phi_i$ of $H$ are conjugate (or twisted) to each other by this automorphism. 
  In particular, the three images $\phi_i(H)$ are isomorphic as abstract groups; one of them is regarded as the natural orthogonal group $\Omega_8^+(q)$ and the other representations are quasi-equivalent to it (see~\cite[\S\,3.9.3]{W}). Moreover, since these representations  are all conjugate under triality to the natural representation we
  do not list them in Table~\ref{t:c9-samechar}. We note that this is the same approach as that used in \cite{BHRD, GPPS} where
  each $\Omega_8^+(q)$ is regarded as a `geometric subgroup' of
  $\GL_8(q)$, as discussed in Section~\ref{s:generallemma},  even though there may exist several
  conjugacy classes of subgroups $\Omega_8^+(q)$ in $\SL_8(q)$ not all conjugate in $\GammaL_8(q)$ (see \cite[Table 8.44]{BHRD}). For $q$ even, $H=\Spin_{8}^{\epsilon}(q)\cong \Omega_8^+(q)$, and again we have three representations of $H$ of degree $8$; they are quasi-equivalent under triality and have images $\Omega_8^+(q)$, \cite[Proposition 5.4.9]{KL}.

Finally, we make some comments about stingray elements in these subgroups. As already discussed the natural subgroup $\Omega_8^+(q)$ of $\GL_8(q)$ contains $4$-ppd stingray elements of order $r$ for each $4$-ppd prime $r$ dividing $q^4-1$. Observe that, for an element $h\in H$ of $4$-ppd prime order $r$,  if $\phi$ is some representation of $H$ and
  $\phi(h)$ is a $4$-ppd stingray element of order $r$, then also
  $\phi^*(h)$ and $\phi'(h)$ are $4$-ppd stingray elements of order $r$, where $\phi^*$ is the dual of $\phi$, and $\phi'$ is a
  Galois conjugate of $\phi$.  However, distinct `spin' representations $\phi_i$ and $\phi_j$ are `quasi-equivalent' under triality, and the Jordan forms of $\phi_i(h)$ and $\phi_j(h)$ may be different; in particular one may be a $4$-ppd stingray element but not the other.  We give a small example for $H=\Spin_{8}^{\epsilon}(2)\cong \Omega_8^+(2)$. There are three $\Omega_8^+(2)$-classes of elements of order $5$, called $5A, 5B, 5C$, and we note that $r=5$ is a $4$-ppd prime divisor of $q^4-1=15$. These three classes are fused under triality. Also $\Omega_8^+(2)$ has three irreducible representations of degree $8$, and for each of the classes $5A, 5B, 5C$, the Brauer character values of the three representations are $3,-2,-2$ (in a different order for the different classes), see \cite[p. 232]{Modat}. It follows from our discussion after Remark~\ref{r:e-ppd} that the image $\phi(h)$ of an element $h\in H$ of order $5$ under an irreducible representation $\phi$ of degree $8$ is a $4$-stingray element if and only if its Brauer character value is $8/2-1=3$, and otherwise it has character value $-2$. Thus, for each element $h\in H$ of order $5$, exactly one of the three images of $h$ under these three representations is a $4$-ppd stingray element; and also in each of the three irreducible degree $8$ representations of $H$ there is exactly one class of  $4$-ppd stingray elements of order $5$.\qed  
\end{remark}

\section{Quasisimple groups involving alternating groups }\label{s:alt}

In this section we consider the case where $G$ is quasisimple containing a $(d/2)$-ppd element, and $S\coloneq G/\Z(G)$ is an alternating group.  We treat in Subsection~\ref{ss:perm} the case where the natural module $V=\F_q^d$ for $\GL_d(q)$ is  
the fully deleted permutation module for $S_n$ over $\F_q$, and consider all other cases in Subsection~\ref{ss:alt}.

\subsection{Alternating groups and fully deleted permutation modules}\label{ss:perm}

Here $S=A_n$ and  $V$ is  the fully deleted permutation module for $S_n$ over $\F_q$, which we define as follows. The permutation module for $S_n$ over $\F_q$ is the space  $Y=\F_q^n=\langle y_1,\dots, y_n\rangle$ such that $S_n$  permutes the basis vectors $y_1,\dots, y_n$ naturally. There are two invariant subspaces:
\[
W=\left\{ \sum_{i=1}^n x_iy_i \mathrel{\Big|} \sum_{i=1}^n x_1=0\right\}\quad \mbox{and}\quad D=\langle e\rangle, \ \mbox{where}\ e=\sum_{i=1}^n y_i.
\]
If $w_i=y_i-y_{i-1}$ for $1\le i<n$, then $w_1,w_2,\dots,w_{n-1}$ is a basis for $W$, and it follows from $w_1+2w_2+\cdots+(n-1)w_{n-1}=y_1+\cdots+y_{n-1}-(n-1)y_n$ that $D\le W$ if and only if $p\mid n$.
The \emph{fully deleted permutation  module} $V$ is the quotient space $V=W/(W\cap D)$, where $W\cap D=D$ if $p\mid n$ and $W\cap D=0$ if $p\nmid n$. Let $\delta=1$ if $p\nmid n$ and $\delta=2$ if $p\mid n$.

\begin{enumerate}
  \item[(PM1)] The following are equivalent: $p\nmid n$,\quad $\delta=1$,\quad
  $Y=W\oplus D$,\quad
  $V\cong W\cong Y/D$;
  \item[(PM2)] The following are equivalent: $p\mid n$,\quad $\delta=2$,\quad
  $D\le W$,\quad $V=W/D$.
\end{enumerate}
The $S_n$-module $V$ is absolutely irreducible, and remains irreducible when restricted to $A_n$. It can be written over the prime field $\F_p$ of $\F_q$. In the sequel, we view $V$ as a $G$-module where $A_n\lhdeq G\le N_{\GL_d(q)}(A_n)$.
We shall assume that $n\ge5$, so that $A_n$ is simple. If $n\ne 6$, then $\Aut(A_n)\cong S_n$, and so $N_{\GL_d(q)}(A_n)=S_n\times Z_{q-1}$. However, $\Aut(A_6)\cong S_6.2$ and $N_{\GL_d(q)}(A_6)=S_6\times Z_{q-1}$ if $p\ne3$, and if $p=3$, then $N_{\GL_4(q)}(A_6)=(A_6\times2):(2\times2)$ and $N_{\GL_4(q)}(A_6)/A_6\cong D_8$. In all cases $A_n \leq \GL_d(p)$, see \cite[p.\,187]{KL}.

First we show the various ways in which this action may involve $(d/2)$-ppd elements and $(d/2)$-ppd stingray elements. To explain the third column of Table~\ref{t:c9-permmod}, a permutation $x\in S_n$ with $a_i$ cycles of length $i$, where $n=\sum_ia_ii$, has \emph{cyclic type} $1^{a_1}\dots, n^{a_n}$; and we usually omit any factors with $a_i=0$.

\begin{proposition}\label{p:delperm}
Let $V=W/(W\cap D)$ be the deleted permutation module over $\F_q$ for $A_n$,  where $n\ge5$ and $d=\dim V=n-\delta$ as above. Then $H\coloneq N_{\GL_d(q)}(A_n)$ is $S_n\times Z_{q-1}$, unless $n=6, p=3$, when $H=(A_6\times 2): 2^2$; in all cases $A_n\leq \GL_d(p)$. Let $g\in H$ have order a ppd prime divisor $r$ of $q^{d/2}-1$. Then $g\in A_n$ and one of the lines of Table~$\ref{t:c9-permmod}$ holds. Moreover,  $g$ is a $(d/2)$-stingray element if and only if the entry in the last column of Table~$\ref{t:c9-permmod}$ is $\checkmark$. Thus if the hypotheses of Theorem~$\ref{t:stingray}$ hold, then $q=p, G= A_n$, and one of the lines of Table~$\ref{t:c9-permmod}$ holds.
\end{proposition}

\begin{proof}
As mentioned above the representation is realisable over the prime
subfield so $A_n\leq \GL_d(p)$, and as $A_n$ is absolutely
irreducible on $V$ the group $H=N_{\GL_d(q)}(A_n)$ is as in the
statement.  Suppose that $r$ is a ppd prime divisor of $q^{d/2}-1$
and that $r=|g|$ for some $g\in X$. Then $2<r\leq n$ and $r$ does
not divide $q-1$, and hence $g\in A_n$. Also $r=k(d/2)+1$ for some
$k\geq1$. Since $d=n-\delta$, we have $1+ d/2 \leq r= 1+k(d/2)\leq
n=d+\delta \leq d+2$. Thus $k\leq 2$. Suppose first that $k=2$ so
$r=d+1$. Then either $n=d+1$ (so $p$ does not divide $n$) and $g$
is an $n$-cycle, or $n=d+2$ (so $p$ divides $n$) and $g$ is an
$(n-1)$-cycle, as in lines 1,2 of
Table~\ref{t:c9-permmod}. Now suppose that $k=1$ so $r=d/2+1 =
(n+2-\delta)/2 \geq n/2$ and $r= (n+2-\delta)/2 \leq n-2$ (since
$n\geq5$). Then either $g$ is an $r$-cycle with $n-r$ fixed
points, or $\delta=2$ (so $p$ divides $n$) and $r=n/2$ and $g$ is
a product of two $r$-cycles. In the latter case, since $r$ is
coprime to $p$, and as $n=2r$ is divisible by $p$, we must have
$p=2$. Thus in lines $3, 4, 5$ of Table~\ref{t:c9-permmod}
the values of $r,n,p$ and the cycle type of $g$ are as shown.

We now decide which of these elements is a stingray
element. Consider first lines~$1, 2$ with $r=d+1$. If $g$ is
an $n$-cycle and $p\nmid n$, then $g$ has characteristic
polynomial $t^n-1$ on $Y$ and $(t^n-1)/(t-1)$ on $V\cong
Y/D$. Similarly, if $g$ is an $(n-1)$-cycle and $p\mid n$, then
$g$ has characteristic polynomial $(t^{n-1}-1)(t-1)$ on $Y$,
$t^{n-1}-1$ on $W\cong Y/D$, and $(t^{n-1}-1)/(t-1)$ on
$V=W/D$. In both cases the characteristic polynomial of $g$ on $V$
evaluated at 1 is non-zero. Hence $g$ is not a
$(d/2)$-stingray element. Consider next lines $3,4$ of
Table~\ref{t:c9-permmod} where
$r=d/2+1$ and $g$ is an $r$-cycle. Here $n/2\leq r\le
n-2$. Since $r$ is a ppd prime, the polynomial
$f(t)=t^{r-1}+\cdots+t+1$ is irreducible over $\F_q$. Now $g$ has
characteristic polynomial $f(t)(t-1)^{d+1-r}$ on $Y$. If $p\nmid
n$, its characteristic polynomial is $f(t)(t-1)^{d-r}$ on $V\cong
Y/D$. Also if $p\mid n$ then a direct calculation shows that $g$
has characteristic polynomial $f(t)(t-1)^{d-r}$ on $W$, and
therefore $f(t)(t-1)^{d-1-r}$ on $V=W/D$. Hence $g$ is a stingray
element on $V$ in these cases.  Finally, if $r=d/2+1$ and $g$ is a
product of two $r$-cycles, then $g$ has characteristic polynomial
$(t^r-1)^2$ on $Y$, and hence the characteristic polynomial of $g$
on $V$ has a repeated eigenvalue which is not~1. Thus $g$ is not a
stingray element. This verifies the entries in the last column of
Table~\ref{t:c9-permmod}, and completes the proof.
\end{proof}

The final paragraph of Proposition~\ref{p:delperm} could have been proved using \cite[Lemmas~4.4 and~4.3]{PSY}; however, it is not a straight-forward application of these very general lemmas.

\subsection{Other actions of alternating groups or their covering groups}\label{ss:alt}

Here we use  \cite[Theorem 1.1, Table 1]{DiM} to determine the remaining quasisimple groups with alternating factors which contain $(d/2)$-ppd elements. Note that $G$ is either $A_n$ or a covering group of $A_n$.

\begin{proposition}\label{p:alt2}
 Suppose that $d, q, V$ and $G<\GL_d(q)$  satisfy the hypotheses of Theorem~$\ref{t:stingray}$ such that $S = G/\Z(G)=A_n$ but that $V$ is not the deleted permutation module for $S$ over $\F_q$. Suppose also that $G$ contains an element of order a $(d/2)$-ppd prime $r$. Then $G, d, r$ are as in one of the lines of Table~$\ref{t:c9-alt}$. Moreover, $G$ contains a $(d/2)$-ppd stingray element of order $r$ precisely when there is a $\checkmark$ in the last column of Table~$\ref{t:c9-alt}$.
\end{proposition}

\begin{proof}
By \cite[Theorem 1.1]{DiM}, one of the lines of 
\cite[Table 1]{DiM} holds. Now $r\nmid q-1$ implies $r\mid |S|$, $r\leq n$, and also $r=k(d/2)+1 \geq d/2 + 1$. 
Each group $S=A_n$ arising in \cite[Table 1]{DiM} has $n\leq 14$ and hence $1+d/2\leq r\leq 13$, so $d\leq 24$. Moreover, if $r$ is 11 or 13, then $n\geq 11$, and by \cite[Table 1]{DiM}, the only possibilities have $d=16$ or $32$, and since $r=k (d/2)+1 \in\{11, 13\}$, we see that this is not possible. Hence $r\leq 7$ and the only possibilities in \cite[Table 1]{DiM} for $r=k (d/2)+1$ and $d$ are: (i) $r=3, d=4$; (ii) $r=5, d=4$ or $8$; or (iii) $r=7, d= 4, 6$ or $12$. 
We show that one of the lines of Table~\ref{t:c9-alt} holds. 

If $(r,d)=(7,12)$,  then the only possibility in \cite[Table 1]{DiM} is $G=6.A_7$ with $p=5$. Then $6=|\Z(G)|$ divides $q-1=|\Z(\GL_d(q))|$. Thus $q=5^a\equiv (-1)^a\equiv 1\pmod{6}$, and hence $a$ is even.  Now  $o_7(q)=d/2=6$, but $o_7(5^a)$ divides $o_7(5^2)=3$, which is a contradiction.

    If $(r,d)=(5,8)$,  then $p\ne5$ and the possibilities in \cite[Table 1]{DiM} are precisely the groups in lines 7 and 8 of  Table~\ref{t:c9-alt}, together with the following possibilities:
    \begin{center}
        $G=A_6$ with $p\ne 3$, or $G=\nonsplit{2}{A_6}$ with $p\geq7$.
    \end{center} 
    Since $o_5(q)=d/2=4$, we have  $p\equiv 2, 3\pmod{5}$ and $a$ is odd. In line 7 of Table~\ref{t:c9-alt}, $G=A_9$, and the table \cite[Table 1]{DiM} gives $p=2$. A {\sf GAP} computation with  the AtlasRep package shows that $G=A_9$ is an absolutely irreducible subgroup of $\GL_8(2)$ and hence we have $q=2$. (This $A_9$ action is not the fully deleted permutation action, see \cite[p.\,85]{Modat}.) We show that the groups $G=A_6$ with $p\ne 3$, and $G=\nonsplit{2}{A_6}$ with $p\geq7$ do not give examples. Let $g\in G$ have order $r=5$. By \cite[p.\,5]{Atlas} and \cite[p.\,4]{Modat}, the Brauer character value of $g$ is $-b5$ where $b5=z+z^4$ with $z^5=1\neq z$ (see \cite[p.\,xxii]{Atlas}).  Since $p\ne5$, $g$ is conjugate over the algebraic closure of $\F_q$ to a diagonal matrix having entries $z^i$ with multiplicity $c_i$ for $0\leq i\leq 4$, and with $\sum_{i=0}^4 c_i=8$. Thus the character value satisfies $\sum_{i=0}^4 c_iz^i = -z-z^4$, that is to say, 
    \[
    (c_1+1)z + c_2 z^2 + c_3 z^3 + (c_4+1)z^4 = -c_0,\quad \text{with}\quad \sum_{i=0}^4 c_i=8,
    \]
    and it follows from Lemma~\ref{l:rootsC} that $(c_0,c_1,c_2,c_3,c_4)=(2, 1, 2, 2, 1)$, which conflicts with each of the possibilities for a $4$-ppd element $g$ of order $5$ given in Remark~\ref{r:e-ppd}. (In fact one may deduce that  $g$ is a $2$-ppd element with a $2$-dimensional fixed point subspace in $V$ and three $2$-dimensional irreducible composition factors, as in Definition~\ref{d:ppd}.)

    If $(r,d)=(7,6)$, then by \cite[Table 1]{DiM} one of the
      lines 4, 5 or 6 of Table~\ref{t:c9-alt} holds. Since $d=6$, we
      must have $o_7(q)=3$, and hence either $p\equiv 3,5\pmod{7}$
      with $\gcd(a,6)=2$, or $p\equiv 2,4\pmod{7}$ with $\gcd(a,3)=1$.
      In either case we have infinitely many examples in lines 5 and
      6. However, in line 4 we have $p=3$ and the requirement is that
      $a$ is even, and a {\sf GAP} computation with the AtlasRep
      package~\cite{AtlasRep} shows that $G=\nonsplit{2}{A_7}$ is an
      irreducible subgroup of in $\GL_6(9)$ and hence we have $q=9$ in
      line 4. 

    If $d=4$ and $r=3, 5$ or $7$, then by \cite[Table 1]{DiM}, one of the following: 
    $G$ is $\nonsplit{2}{A_5}$ or $\nonsplit{2}{A_6}$ with $p\geq 5$ and $r=3,5$ (as in line 1 of
    Table~\ref{t:c9-alt}); or $G=\nonsplit{2}{A_7}$ with either $r=3$ and $p\geq5$, or $r=5,7$ 
    and $p\geq3$ (as in line 2); or $G$ is $A_7$ or $A_8$ with $p=2$.
    We have $A_7< A_8\cong\GL_4(2) \leq \GL_4(q)$, and as $G$ is not
    realisable over a proper subfield, $q=2$, and further, as $G$ is
      not a classical group, $G$ is not $A_8$ and we only have $G=A_7$
      in line 3 of Table~\ref{t:c9-alt}. Also, as $r$ is a primitive
    prime divisor of $q^2-1=3$ it follows that $r=3$, and in this case
    we do find $2$-stingray elements acting nontrivially on a unique
    $2$-subspace and fixing a complement, namely class $3B$ (see
    \cite[pp. 13 and 48]{Modat}).

    Now we determine which of the other groups $G$ in
    Table~\ref{t:c9-alt} contains a $(d/2)$-ppd stingray element $g$
    of prime order $r$. This involves inspecting the Brauer characters
    in \cite{Modat} for these representations where $p$ divides $|G|$,
    and the character tables in \cite{Atlas} for those representations
    where $p$ does not divide $|G|$. If $r=1+k(d/2)$ then a
    $(d/2)$-stingray element $g$ is conjugate in
    $\GL_d(\overline{\F_q})$ to a diagonal matrix ${\rm diag}(\eta_1,
    \eta_2,\dots, \eta_{d/2}, 1,\dots,1)$, where the $\eta_i$ are
    pairwise distinct primitive $r^{\rm th}$ roots of unity in
    $\overline{\F_q}$, an algebraic closure of $\F_q$. The Brauer
    character is the trace of this matrix, namely $\chi(g)=d/2 +
    \sum_{i=1}^{d/2} \eta_i$. We also use Lemma~\ref{l:don2}.

 \medskip\noindent   
\emph{Line $1$.}\quad 
Here $d=4$ and there are $1$ or $2$ representations of dimension $4$ for $n=5$ or $6$, respectively. If $r=|g|=5$, then  $p>5$ and  $\chi(g)=-1$, see \cite[pp. 2 and 5] {Atlas}, so $g$ is not a $2$-stingray element by Lemma~\ref{l:don2}(b).  Hence $r=|g|=3$, and we have $\chi(g)=1$ if $n=5$, and if $n=6$ then for each of the two representations the character value $\chi(g)$ on the two classes of elements of order $3$ is $-2$ or $1$, see \cite[pp. 2 and 5] {Atlas}.  
It follows from Lemma~\ref{l:don2}(a) that $g$ is a $2$-stingray element if $n=5$; and if $n=6$ then this is true for each of the two representations for exactly one of the conjugacy classes of elements of $G$ of order $3$.

 \medskip\noindent   
\emph{Line $2$.}\quad 
Here $d=4$ and there are two representations of degree $4$, see \cite{Modat} and \cite[p. 10]{Atlas}. If $r=|g|=5$ then $\chi(g)=-1$ and by Lemma~\ref{l:don2}(b), $g$ is not a $2$-stingray element. If $r=|g|=7$ then (see \cite[p. 10]{Atlas} and \cite[p.13]{Modat}) $\chi(g)=-b_7$ where $b_7=\omega+\omega^2+\omega^4$ for $\omega=e^{2\pi i/7}$, see \cite[p. xxvii]{Atlas}.  Setting $-b_7 = 2 + \eta_1+\eta_2$ leads to a sum of $7^{\rm th}$ roots of unity violating Lemma~\ref{l:rootsC}. Thus $r=|g|=3$, and there are exactly two distinct primitive $3^{rd}$ roots of unity giving $1=2+\eta_1+\eta_2=\chi(g)$, and exactly one of the two conjugacy classes of elements of order $3$ has character value $\chi(g)=1$ (namely class $3B$), and both representations of degree $4$ give examples.

 \medskip\noindent   
\emph{Lines $4, 5, 6$.}\quad 
Here $d=6$, $G=x.A_7$ with $x\in\{2,3,6\}$, and $r=|g|=7$. 
By \cite[p. 10]{Atlas} and \cite[p.13-14]{Modat}) $\chi(g)=-1$, so $g$ is not a $3$-stingray element by Lemma~\ref{l:don2}(b).

\medskip\noindent   
\emph{Lines $7, 8$.}\quad 
Here $d=8$, $G=\nonsplit{2}{A_8}$ or $\nonsplit{2}{A_9}$,  $r=|g|=5$, and $p\ne 2$. 
By \cite[pp. 22, 37]{Atlas} and \cite[pp. 49, 51]{Modat}) $\chi(g)=-2$,  so $g$ is not a $3$-stingray element by Lemma~\ref{l:don2}(a).
\end{proof}

\section{Quasisimple groups involving a sporadic group }\label{s:spor}

Here we consider the case where $G$ is a quasisimple subgroup of $\GL_d(q)$ containing a $(d/2)$-ppd element of prime order $r$, such that $S\coloneq G/\Z(G)$ is a sporadic simple group. We use \cite[Theorem 1.1, Table 2]{DiM} to determine the possibilities for $G, d, r$, and then, in view of Remark~\ref{r:e-ppd}, we apply results from \cite{DiMPZ} to show that none of these examples contains a $(d/2)$-ppd stingray element.

\begin{proposition}\label{p:c9-spo}
 Suppose that $d, q, V$ and $G<\GL_d(q)$  satisfy the hypotheses of Theorem~$\ref{t:stingray}$ such that $S = G/\Z(G)$ is a sporadic simple group,  and that $G$ contains an element of order a $(d/2)$-ppd prime $r$. Then $G, d, r$ are as in one of the lines of Table~$\ref{t:c9-spo}$,  and none  of the $(d/2)$-ppd elements in these groups are $(d/2)$-ppd stingray elements.
\end{proposition}

\begin{proof}
    By  \cite[Theorem 1.1]{DiM}, one of the lines of  \cite[Table 2]{DiM}  holds. To find the examples containing a $(d/2)$-ppd element of prime order $r$ we use the facts that $r\mid |S|$ and $r\equiv 1\pmod{d/2}$, and also $r\ne p$ while $o_r(q)=d/2$. 

\medskip\noindent
    \emph{Case: $S$ is one of the Mathieu groups.}\quad 
    Suppose first that $S=\Mathieu_{11}$ or $\Mathieu_{12}$. Then the conditions $d$ even and $r\equiv 1\pmod{d/2}$ imply that $(d,r)=(10,11)$, as in line 1 of Table~\ref{t:c9-spo}, with no restrictions on $p$ if $G=\Mathieu_{11}$ (except that $o_{11}(q)=5$), $p\geq3$ if $G=\nonsplit{2}{\Mathieu_{12}}$, and with  $p \in\{2,3\}$ if $G=\Mathieu_{12}$. We consider the last case further. For both $p=2$ and $p=3$ we confirmed using the AtlasRep package~\cite{AtlasRep} in {\sf GAP}~\cite{GAP} that $G=\Mathieu_{12}$ is realisable over $\F_p$, and hence $q=p$. However, if $q=2$ then $o_{11}(q)=10=d$, which is a contradiction, and so $q=3$ in this case.
    
    Next if $S=\Mathieu_{22}$, then the conditions that $d$ is even, $r\equiv 1\pmod{d/2}$, and $r\ne p$, imply that either $(d,r)= (6,7)$ with $G=\nonsplit{3}{\Mathieu_{22}}$ and $p=2$, or $(d,r)= (10,11)$ with $G=\Mathieu_{22}$ and $p=2$, or $G=\nonsplit{2}{\Mathieu_{22}}$ and $p\geq3$. In the first case  a {\sf GAP} computation using the AtlasRep package~\cite{AtlasRep}   shows that the smallest field over which this representation is realisable is $\F_4$ and hence $q=4$, as in line 2 of Table~\ref{t:c9-spo}. If $(d,r)= (10,11)$ with $G=\Mathieu_{22}$, a similar computation shows that $G$ is an irreducible subgroup of $\GL_{10}(2)$ and hence $q=2$. However this implies that $o_{11}(q)=10=d$, which is a contradiction, so in this case we only have the group $G=\nonsplit{2}{\Mathieu_{22}}$ with $p\geq3$, and this satisfies line 3 of Table~\ref{t:c9-spo}.

    Now suppose that $S=\Mathieu_{23}$ or $\Mathieu_{24}$.
    The conditions $d$ even and $r\equiv 1\pmod{d/2}$ imply that $(d,r)=(22,23),$ or $(44,23)$. However if $d=44$ then by \cite[Table 2]{DiM} the prime $p=2$, and we have $2^{11}\equiv 1\pmod{23}$. Thus $o_{23}(2^a)$ is never equal to $d/2=22$, and so $(d,r)=(22,23)$. Moreover, by \cite[Table 2]{DiM}, either $G=\Mathieu_{23}$ with $p\geq 3$, or $G=\Mathieu_{24}$ with $p=3$.
    In the latter case we confirmed using the AtlasRep package~\cite{AtlasRep} that $G=\Mathieu_{24}$ is realisable over $\F_3$, and hence $q=3$. Thus Table~\ref{t:c9-spo} line 4 holds.

    \medskip\noindent
    \emph{Case: $S$ is one of the other sporadic groups.}\quad 
     Checking these groups in \cite[Table 2]{DiM} we find that the only examples for which $d$ is even,  $r\equiv 1\pmod{d/2}$, and $o_r(q)=d/2$, are those in lines 5--13 of Table~\ref{t:c9-spo}. We make several comments about this checking procedure. First, the condition $r\ne p$ while $o_r(q)=d/2$  ruled out a few cases. For example, if $G=\nonsplit{2}{\Sz}<\GL_{12}(3)$ then we cannot have $r=13$ since if $13$ divides $(3^a)^{d/2}-1=3^{6a}-1$ then $13$ already divides $3^{3a}-1$; and similarly  if $G=\nonsplit{3}{\Sz}<\GL_{12}(4)$ then we cannot have $r=7$. Also, if $G=Ru<\GL_{28}(q)$ with $r=29$ and $p=2$, then a {\sf GAP} computation using the AtlasRep package~\cite{AtlasRep}   showed that  this representation is realisable over $\F_2$ and hence $q=2$, whereas $o_{29}(2)\ne 14$. In cases where the characteristic $p$ was constrained to be one of at most two values we found (for all but one group), using the AtlasRep package~\cite{AtlasRep}, the smallest field over which the representation  was realisable, thus determining the possible values of $q$. The group for which we were unable to do this computationally was  $G=Co_1$; by \cite[Table 2]{DiM}, $G\leq \GL_{24}(2^a)$, and here it follows from \cite[Result 4.2 and Proposition 4.3]{R18} that $G$ is an irreducible subgroup of $\GL_{24}(2)$ so that $q=2$, as in line 8 of Table~\ref{t:c9-spo}.

 \medskip\noindent
    \emph{Existence of stingray elements.}\quad 
    Finally we determine which of the groups in Table~\ref{t:c9-spo} contain $(d/2)$-stingray elements, using the results of \cite[Section 7]{DiMPZ}.
    By \cite[Theorem 7.1]{DiMPZ}, for each of the lines of Table~$\ref{t:c9-spo}$ \emph{except the lines $5, 8, 10, 12$}, the $(d/2)$-ppd elements are cyclic and hence are of type (2.i) of Remark~\ref{r:e-ppd}, so are not stingray elements. Moreover, none of the cases in the remaining lines $5, 8, 10, 12$ of Table~$\ref{t:c9-spo}$ occurs in \cite[Theorem 7.2]{DiMPZ}, and hence in these cases the $(d/2)$-ppd elements are not even almost cyclic, and hence again they are not stingray elements (they are of type (2.ii) of Remark~\ref{r:e-ppd}).  
\end{proof}

\section{Quasisimple groups involving a Lie type simple group: cross characteristic case}\label{s:cross}

Now we consider the case where $G$ is a quasisimple subgroup of $\GL_d(q)$ containing a $(d/2)$-ppd element of prime order $r$, such that $S\coloneq G/\Z(G)$ is a  Lie type simple group over a field of characteristic different from $p$. Suppose that $S$ is  not isomorphic to an alternating group. 
 We use \cite[Theorem 1.1, Table 3--5]{DiM} to determine the possibilities for~$G, d, r$, 

\begin{proposition}\label{p:cross}
     Suppose that $d, q, V$ and $G<\GL_d(q)$  satisfy the hypotheses of Theorem~$\ref{t:stingray}$ such that $S = G/\Z(G)$  is a  Lie type simple group over a field of characteristic different from $p$, and that $S$ is  not isomorphic to an alternating group. Suppose also that $G$ contains an element of order a $(d/2)$-ppd prime $r$. Then either $G, d, r$ are as in one of the lines of Table~$\ref{t:c9-cross-u}$ or  $S, d, r$ are as in one of the lines of Table~$\ref{t:c9-cross-i}$.  Moreover, the quasisimple group $G$ in one of the lines of these tables contains a $(d/2)$-ppd stingray element of order $r$ if and only if the entry in the last column is $\checkmark$.
\end{proposition}

\begin{proof}
    By \cite[Theorem 1.1]{DiM}, one of the lines of  \cite[Table 3, 4, or 5]{DiM} holds; \cite[Tables 3 and 4]{DiM} 
  contain individual examples, while~\cite[Table 5]{DiM} contains infinite families of examples. We use the following facts: $d$ is even, $r$ is prime and $r\ne p$, $r$ divides $|S|$, and $o_r(q)=d/2$. Hence $r\equiv 1\pmod{d/2}$.

 \medskip\noindent
    \emph{Examples from  \cite[Tables 3 and 4]{DiM}.}\quad   
  Modulo the following comments we find from \cite[Table 3]{DiM} that the only examples (untwisted groups) are those in lines 2--6 of Table~\ref{t:c9-cross-u}, and from \cite[Table 4]{DiM} that the only examples  (twisted groups) are those in lines 7--9 in  Table~\ref{t:c9-cross-u}. We note that in line 8 of  Table~\ref{t:c9-cross-u}, where $G=2.\PSU_4(2)\cong \Sp_4(3)$ and  $d=4$ with $r\in\{3,5\}$, the prime $p$ is at least $5$, see \cite[Tables 8.9, 8.11, 8.13, 8.14]{BHRD}. In line 2 of Table~\ref{t:c9-cross-u} the group is $\nonsplit{4}{\PSL_3(4)}$ with $d=4$ and $p=3$, and by \cite[Table 3]{DiM} the prime $r$ is $5$ or $7$. A {\sf GAP} computation using the AtlasRep package~\cite{AtlasRep}   shows that the smallest field over which this representation is realisable is $\F_9$ so $q=9$. However $o_7(9)=4=d$, so $r\ne 7$ and only $r=5$ is possible. 
  
  The groups occurring in \cite[Table 3]{DiM}  with $d=6$ are $\nonsplit{2}{\PSL_3(4)}$ with  $p=3$, and $\nonsplit{6}{\PSL_3(4)}$ with $p\ne 2,3$; and in either case $r\in\{5,7\}$. We require $o_r(q)=d/2=3$, and since $o_r(q)$ divides $r-1$ it follows that $r=7$. If $G=\nonsplit{2}{\PSL_3(4)}$ with  $p=3$ then a {\sf GAP} computation using the AtlasRep package~\cite{AtlasRep}   shows that this representation is realisable over $\F_3$ so $q=3$; however $o_7(3)=6\ne d/2$. Thus $G=\nonsplit{6}{\PSL_3(4)}$ as in line 3 of Table~\ref{t:c9-cross-u}.
  
   The groups arising from \cite[Table 4]{DiM} are precisely those listed in lines 7--9 in  Table~\ref{t:c9-cross-u}, except that in line 7 or 9 we only know that $p=5$ or $2$, respectively. In these two cases a  {\sf GAP} computation using the AtlasRep package~\cite{AtlasRep}   shows that $q=5$ or $4$, respectively.

   \medskip\noindent
    \emph{Examples from  \cite[Table 5]{DiM}.}\quad   
  The possibilities from \cite[Table 5]{DiM} require more care in checking and we give some details. We find that the only infinite families of examples are those in Table~\ref{t:c9-cross-i}, with a single example recorded in line 1 of Table~\ref{t:c9-cross-u}.
 The most useful test is that $r\equiv 1\pmod{d/2}$, in particular that $r> d/2$. This test rules out all cases with $S=\PSp_{2n}(s)$ except for line 1 of Table~\ref{t:c9-cross-i}; and all cases with $S=\PSU_{n}(s)$ except for line 2 of Table~\ref{t:c9-cross-i}. For all other examples $S=\PSL_n(s)$. We give some details of the arguments used to treat these cases: 

\smallskip\noindent 
{\it The three lines for `$S=\PSL_n(s), n\geq3$', with $r=(s^n-1)/(s-1)$ and $n$ prime:} Here the values of $r, d$ in  \cite[Table 5]{DiM} imply that $r\in\{d+2, d+1, d\}$, and hence $r=d+1$ as in line 3 of Table~\ref{t:c9-cross-i}.

\smallskip\noindent 
{\it The eleven lines with $S=\PSL_2(s)$, $s\geq 7$, $s\ne 9$:} Some of these lines are a bit tricky to deal with, for example:
    
    (a) In the $6^{\rm th}$ line for this case,  we have $r=(s+1)/3$ and $d\in\{(s\pm 1)/2, s-1, s\}$. Now $r<2\cdot (d/2)+1$ for all these values of $d$ and hence $r=d/2+1$. This implies that $d=(s\pm 1)/2$, and then, solving $(s+1)/3 = 1+(s\pm 1)/4$ yields $s= 5$ or $11$, and $r=2$ or $4$, which is not an odd prime.
    
    (b) In the last five lines of this case we have $d\in\{(s\pm 1)/2, s\pm 1, s\}$ with $s\geq7$ and $r\in\{(s\pm 1)/2, s\pm 1, s\}$.  If $d=s+1$, then since $d$ is even and $r-1$ is divisible by $d/2\geq4$, there are no possibilities for $r$. If $d=s-1$ then, again since $d$ is even and $r-1$ is divisible by $d/2\geq3$, we see that $r$ is either $s$ or $(s+1)/2$, and  by \cite[Table 5]{DiM}, either $s$ is prime, or $r=(s+1)/2$ with $s=t^{2^{c}}$, $t$ an odd prime and $c\geq1$, as in line 5 of  Table~\ref{t:c9-cross-i}.  Next suppose that $d=s$. Then since $d$ is even, it follows from \cite[Table 5]{DiM} that $s=2^c$ for some $c\geq 3$ (since $s\geq7$), and then, since $r-1$ is divisible by $d/2 = 2^{c-1}\geq4$, we see that $r=s+1$, and in particular $s\ne 8$ since $r$ is prime. Checking   \cite[Table 5]{DiM} again, we have $s=2^{2^c}$ with $r$ a Fermat prime, as in line 4 of Table~\ref{t:c9-cross-i}. Similar arguments for the case where $d=(s- 1)/2$ lead to the examples in line 6 of Table~\ref{t:c9-cross-i} (recall that $S\not\equiv L_2(9)$). Finally if $d=(s+1)/2$ then we obtain only the possibility $(d,s)=(4,7)$ as in the first line of Table~\ref{t:c9-cross-u}; in this case the group $G$ must be $\nonsplit{2}{\PSL_2(7)}$ by \cite[p.\,3]{Atlas} and \cite[p.\,3]{Modat}.

\smallskip
    Thus we have determined all possibilities for $S, d, r$.

     \medskip\noindent
    \emph{Existence of stingray elements.}\quad 
    Now we determine whether or not these examples provide $(d/2)$-stingray elements. First we consider the cases where $S=\PSL_2(s)$.

\medskip\noindent   
\emph{Claim $1$. If $S=\PSL_2(s)$ with $s$ coprime to $q$, then $G$ contains a $(d/2)$-ppd stingray element of order $r$ if and only if $s=7, d=4, r=3$, and line $1$ of Table~$\ref{t:c9-cross-u}$ holds.}

Note that, since $S$ is not isomorphic to an alternating group, $s\geq 7$ and $s\ne 9$.
The cases to consider are  line $1$ of Table~\ref{t:c9-cross-u} (with $s=7$ and $d=4$), and  lines 4--6 of Table~\ref{t:c9-cross-i} (with $d\in\{s, s-1, (s-1)/2\}$).  In all these cases $G$ is the image of an irreducible representation $\phi:\SL_2(s)\to \GL_d(q)$. Then, for example by \cite[Lemma 4.1]{DiMZ}, this representation lifts to a representation  over the complex numbers, so it suffices to deal with ordinary representations of $G$. Suppose that $g'\in\SL_2(s)$ is such that $g\coloneq \phi(g')$ is a $(d/2)$-stingray element. Then $1$ occurs as an eigenvalue of $g$ with  multiplicity  $d/2$, and the number of pairwise distinct eigenvalues of $g$, denoted $\deg(g)$, is $d/2+1$ (see Remark~\ref{r:e-ppd}(1)). Moreover,  $g$ is in particular almost cyclic. If $g$ is a semisimple element of $G$, then by \cite[Lemma 4.3]{DiMZ}, $d\leq r+1=|g|+1$ and $r=(s\pm 1)/(2,s+1)$, and 
these restrictions imply that $\deg(g)=d/2+1$ is at most $(r+3)/2$. On the other hand, if $g$ is not a semisimple element of $G$, then $r=s$ and one of $(d,r)=(4,7)$ (line $1$ of Table~\ref{t:c9-cross-u}), or $(d,r)=(s-1,s)$ with $s\geq7$ (line 5 of Table~\ref{t:c9-cross-i}),  or $(d,r)=((s-1)/2,s)$ with $s\geq 11$ (line 6 of Table~\ref{t:c9-cross-i}); and in each of these cases $d\leq r+1$ and $d/2+1\leq (r+3)/2$ also hold. 

Suppose that $r\geq 7$, and note that this always holds if $g$ is not a semisimple element of $G$. Then $d/2+1\leq (r+3)/2<r-1$, and hence by \cite[Theorem 1.2]{Z06}, since $s\geq7$ and $s\ne 9$, one of the cases (A), (B)(1) (with $k=1$) or (C) of that theorem holds, and in each of these cases $\deg(g)$ is equal to $d$, which is a contradiction since $\deg(g)=d/2+1$ and $d>2$. Thus $r\leq 5$, and in particular $g$ is a semisimple element of $G$. Therefore $d\leq r+1\leq 6$ and  $5\geq r=(s\pm 1)/(2,s+1)\geq (s-1)/(2,s-1)$, and since $s\ne 4,5,9$, it follows that $s=7$ or $11$, with $r=3$ or $5$, and $d\leq 4$ or $d\leq 6$, respectively. In the former case, by  \cite{Atlas}, we have $G=\SL_2(7)<\GL_4(q)$ with $r=d/2+1=3$ and the Brauer character value $\chi(g)=1=d/2-1$ so $g$ is indeed a $2$-stingray element by Lemma~\ref{l:don2}(a) (and $g$ is in the unique class 3A of $3$-elements). In the latter case, since $d\leq 6$ and  $d/2$ divides $r-1=4$, we conclude that $d=4$. However the values $(s,d,r)=(7,4,5)$ do not occur in either  line $1$ of Table~\ref{t:c9-cross-u} or in lines 4--6 of Table~\ref{t:c9-cross-i}. 
Thus Claim 1 is proved.

 \medskip\noindent   
\emph{Claim $2$. None of the $(d/2)$-ppd elements in Table~$\ref{t:c9-cross-i}$ are stingray elements.}\quad    

 By Claim 1 we may assume that $G, d, r$ are as in line 1, 2 or 3 of Table~$\ref{t:c9-cross-i}$.  Suppose that $g\in G$ is a $(d/2)$-stingray element of order $r$. In particular $g$ is almost cyclic  of $(d/2)$-ppd prime order $r$.
     In each case $r=d+1$ and, embedding $\GL_d(q)$ in $\GL_d(\overline{\F_q})$ we may consider the representation over the algebraic closure $\overline{\F_q}$ of $\F_q$. Also in each case $r$ is an odd prime so $d=r-1$ is even. In line 1 we have $S=\PSp_{2n}(s)$ with $s$ odd and $n$ even, and with $d=(s^n-1)/2$ and  $r=(s^n+1)/2$. Thus $d$ 
       is even, and we note that $d$ is the minimum dimension of an irreducible representation of $\SL_2(s^n)$ (in characteristic different from $s$) and that the group
      $\Sp_{2n}(s)$ and its subgroup $\Sp_{2}(s^n)\cong \SL_2(s^n)$ share a common centre. Therefore we may consider $g\in H\coloneq \SL_2(s^n)<G = \Sp_{2n}(s) <\GL_d(q)$ if $q$ is odd, or  $g\in H\coloneq \PSL_2(s^n)<G = \PSp_{2n}(s) <\GL_d(q)$ if $q$ is even.  As the action of $H$ is faithful and $d$ is the minimal degree of a nontrivial irreducible representation of $H$ over $\F_q$ containing a $(d/2)$-ppd element, it follows that line 5 of Table~\ref{t:c9-cross-i} holds with $\PSL_2(s^n), s^n$ in place of $S,s$. Thus $g$ is not a stingray element by Claim 1. 
      
      Finally, consider lines 2 and 3 with $S=\PSU_n(s)$ or $S=\PSL_n(s)$, $n$ an odd prime. In both cases, $d=(s^n-s)/(s-\delta)$ and $r=(s^n-\delta)/(s-\delta)$, where $\delta=1$ or $-1$, respectively. Then a Sylow $r$-subgroup of $G$ is cyclic and, as $d= r-1$, the action of $g$ on $V$ has at most $r-1$ distinct eigenvalues. These groups occur in \cite[Theorem 2.1]{Z08} (part (6) and (1), respectively), and from that result,  all nontrivial $r$th-roots of unity occur as eigenvalues of $g$ over $\overline{\F_q}$. Thus $g$ acts fixed point freely on $V$, and is therefore not a stingray element, which is a contradiction. This completes the proof of Claim 2.

 \medskip\noindent   
\emph{Claim $3$. A line of Table~$\ref{t:c9-cross-u}$  corresponds to a $(d/2)$-ppd stingray element if and only if it contains a $\checkmark$ in the final column. }\quad    

By Claim 1, the assertion holds for line 1 of  Table~$\ref{t:c9-cross-u}$ where we have $S=\PSL_2(7)$. Apart from this line there are exactly five groups  in Table~$\ref{t:c9-cross-u}$ for which the prime $r=d/2+1$ occurs: namely from lines 4, 5 (two groups), 6 (with $r=7$) and 8 (with $r=3$). 
In these cases, by Lemma~\ref{l:don2}(a), $g$ is a $(d/2)$-stingray element if and only if the Brauer character $\chi(g)$ is equal to $d/2-1$, namely 
$3, 3, 5, 1$, respectively.
Checking the character tables in \cite{Atlas} and \cite{Modat} carefully reveals that these values occur for precisely two cases, giving examples for line 5, $G=\nonsplit{2}{{\rm P}\Omega_8^+(2)}$, and line 8, $G=\nonsplit{2}{\PSU_4(2)}\cong\Sp_4(3)$, and  in each case the required character value holds for a unique conjugacy class of $r$-elements, namely class $5A$ and $3D$, respectively. Thus Claim 3 is proved for these five cases. 

The next set of seven groups we consider from Table~$\ref{t:c9-cross-u}$ all have the prime $r=d+1$: namely lines 2, 3, 6 (with $r=13$), 8 (with $r=5$), and 9 (three groups). 
Checking the character tables in \cite{Atlas} and \cite{Modat} carefully reveals that in all these cases  the Brauer character value $\chi(g)$ is equal to $-1$, and it therefore follows from Lemma~\ref{l:don2}(b) that $g$ is of type (2.i) and in particular is not a $(d/2)$-stingray element. 
Thus Claim~3 is proved for these eight cases. 

There is only one case remaining to be treated, namely line 7 and here $d=8$, $q=5$, and $r=13=3d/2+1$:
by \cite{Modat} we find that $\chi(g)=-1-c_{13}$ (in `Atlas notation') or a conjugate of this quantity. However, $c_{13}=\frac{1}{3}\sum_{k=1}^{12}\eta^{k^3}=\eta+\eta^{-1}+\eta^5+\eta^{-5}$ by \cite[p.\,xxvii]{Atlas}, where $\eta$  is a primitive $13^{\rm th}$ root of unity. Since  $\chi(g)=-1-c_{13}$,
the following hold:
\begin{align*}
\chi(g)  &=\chi(g^{-1})=\chi(g^5)=\chi(g^{-5}) =-1-\eta-\eta^{-1}-\eta^5-\eta^{-5}, \quad \text{so also}\\
\chi(g^2)&=\chi(g^{-2})=\chi(g^3)=\chi(g^{-3})=
-1-\eta^2-\eta^{-2}-\eta^3-\eta^{-3},\\
\chi(g^4)&=\chi(g^{-4})=\chi(g^6)=\chi(g^{-6})=-1-\eta^4-\eta^{-4}-\eta^6-\eta^{-6}.
\end{align*}
Noting that $\chi(1_G)=8$, we compute the inner product of the character $\chi|_{\langle  g\rangle}$ of $\langle  g\rangle$ with the trivial character (which equals to the multiplicity of the trivial character in $\chi$) and obtain $(8+4\chi(g)+4\chi(g^2)+4\chi(g^4))/13$, 
which equals $(-4-4\sum_{i=1}^{12} \eta^i)/13 =0$. It follows that $g$ is not a $4$-stingray element.  

This completes the proof of Claim 3, and hence Proposition~\ref{p:cross} is proved.
\end{proof}

\section{Quasisimple groups  involving a Lie type simple group: same characteristic case}\label{sec:samechar}

Finally we consider the case where $G$ is a quasisimple subgroup of $\GL_d(q)$ containing a $(d/2)$-ppd element of prime order $r$, such that $S\coloneq G/\Z(G)$ is a  Lie type simple group over a field of characteristic  $p$. Suppose that $S$ is  not isomorphic to an alternating group, and that  $G$ contains an element of order a $(d/2)$-ppd prime. The only analysis available for these groups is in the PhD thesis of DiMuro~\cite{DiMThesis}  given for $e$-ppd elements of \emph{prime power order with $e$ between $d/3$ and $d/2$} (see especially Tables 1.5--1.7 of the thesis, but the tables do not specify the value of $e$). In Subsection~\ref{s:same-ppd} we give an independent analysis based on arguments in the proofs in \cite[Section 7]{GPPS} which gives a list of possibilities (Table~\ref{t:c9-samechar}) for $e=d/2$ containing all the possibilities in \cite{DiMThesis} along with several additional groups, see Remark~\ref{r:missing}. Then in Subsection~\ref{s:same-stingray} we examine the cases in the list to determine which of these groups contains a $(d/2)$-ppd stingray element of prime order.

\begin{proposition}\label{p:same}
     Suppose that $d, q=p^a, V$ and $G<\GL_d(q)$  satisfy the hypotheses of Theorem~$\ref{t:stingray}$ such that $S = G/\Z(G)$  is a  Lie type simple group over  the field $\F_{p^b}$, and that $S$ is  not isomorphic to an alternating group. Suppose also that $G$ contains an element of order a $(d/2)$-ppd prime $r$. Then $S, d, q$ are as in one of the lines of Table~$\ref{t:c9-samechar}$.  
     Moreover, the quasisimple group $G$ in one of the lines of this table contains a $(d/2)$-ppd stingray element of order $r$ if and only if the entry in the last column is $\checkmark$.
\end{proposition}

\begin{remark}\label{r:missing}
      Table~\ref{t:c9-samechar} contains every example for Proposition~\ref{p:same} of groups arising from \cite[Tables 1.5 to 1.7]{DiMThesis} involving a $(d/2)$-ppd element.  Moreover Table~\ref{t:c9-samechar} contains some possibilities which were mistakenly omitted from the tables in \cite{DiMThesis}, namely the groups in  Table~\ref{t:c9-samechar-tocheck}. In the column `Comments' we give references to confirm these examples.\qed
\end{remark}

  \begin{table}
    \caption{Lie type groups in characteristic $p$ in Table~\ref{t:c9-samechar} missing from  \cite[Table 1.7]{DiMThesis}}
    \vskip2mm
\begin{tabular}{clcl}
  \toprule
Line &$S$ & $d$  & Comments \\
    \midrule
3 &$\PSU_m(q)$   & $28, 28, 36$  & $(m,V)=(7,S^2V_0), (8,\Lambda^2V_0), (9,\Lambda^2V_0)$; see \\
&&&\cite[Prop. 5.4.11, Table 5.4.A]{KL}, or \cite[App. A.10--A.12]{Lu} \\
$9$ &$\POm^-_8(q^{1/2})$   & $8$  & in $\Omega_8^+(q)$; see \cite[[Prop. 5.4.11, Table 5.4.A]{KL} \\
    &   &   &  or \cite[App. A.41]{Lu}, or~\cite[p.\;403]{BHRD}\\
\bottomrule
\end{tabular}
\label{t:c9-samechar-tocheck}
\end{table}

\subsection{Determining the list of quasisimple groups containing a \texorpdfstring{$(d/2)$}{}-ppd element for Proposition~\ref{p:same}}\label{s:same-ppd}

In this subsection we prove the first assertion of Proposition~\ref{p:same}. 
Let $d, q=p^a, V$ and $G<\GL_d(q)$  satisfy the hypotheses of Theorem~$\ref{t:stingray}$, so that $G$ is quasisimple with $S = G/\Z(G)$  a  Lie type simple group $S=\Lie(q_0)$ over  a field $\F_{q_0}$, where $q_0\coloneq p^b$ for some~$b$, and  $S$ is  not isomorphic to an alternating group.  Suppose also that $G$ contains an element $g$ of order a $(d/2)$-ppd prime $r$, so $r=kd/2+1$ for some integer $k$; and  if $r=3$ then $d=4$.
Our analysis is influenced by the arguments in \cite[Section 7]{GPPS}.  

On examining the order of $S=\Lie(q_0)$, see \cite[Tables 5.1.A and 5.1.B]{KL}, we find  that \emph{$r$ is a divisor of  $q_0^x-1$ for some positive integer $x$, and the least such $x$ is at most $\ell$, with $\ell$ as in Table~\ref{t:c9-samechar2}}. Note that in all cases $\ell\geq 2$, and since $r$ does not divide $q-1$, also $x\geq2$. We also record in Table~\ref{t:c9-samechar2} the integer $R_p(S)$, which is the least positive integer $n$ such that $S$ embeds in $\PGL_n(\F)$ for some field $\F$ of characteristic $p$. Since $S$ is not an alternating group, these values are given by \cite[Proposition 5.4.13, see Table 5.4.C]{KL}. Note that either $R_p(S)\geq \ell$, or $S=\PSU_m(q_0)$, or $S={}^3\kern -2pt D_4(q_0)$. 

 \begin{table}
\caption{Properties of $S$ for the proof of Proposition~\ref{p:same}}
    \vskip2mm
\begin{tabular}{lccll}
  \toprule
Type & $S$ & $\ell$ & $R_p(S)$ & Conditions\\
    \midrule
 $A$& $\PSL_m(q_0)$ &   $m$    & $m$    &$m\geq2$\\
 ${}^2\kern-2pt A$& $\PSU_m(q_0)$ &   $2m$    & $m$    &$m\geq3$\\
 $B$& $\POm_m^\circ(q_0)$ &   $m-1$    & $m$    &$m\geq7$, $m, q_0$ both odd\\
 $C$& $\PSp_m(q_0)$ &   $m$    & $m$    &$m\geq4$, $m$ even\\
 $D$, ${}^2\kern-2pt D$ & $\POm^\pm_m(q_0)$ &   $m$    & $m$    &$m\geq8$, $m$ even \\
${}^2\kern-2pt B_2$& ${}^2\kern-2pt B_2(q_0)$ &   $4$    & $4$    &$q_0=2^b$, $b$  odd, $b>1$\\
${}^3\kern-2pt D_4$& ${}^3\kern-2pt D_4(q_0)$ &   $12$    & $8$    &\\
$F_4$& $F_4(q_0)$ &   $12$    & $26-\delta_{p,3}$    &\\
${}^2\kern-2pt F_4$& ${}^2\kern-2pt F_4(q_0)'$ &   $12$    & $26$    &$q_0=2^b$, $b$ odd\\
$G_2$& $G_2(q_0)'$ &   $6$    & $7-\delta_{p,2}$    &\\
${}^2\kern-1pt G_2$& ${}^2\kern-1pt G_2(q_0)'$ &   $6$    & $7$    &$q_0=3^b$, $b$ odd\\
$E_6$, ${}^2\kern-2pt E_6$& $E_6(q_0), {}^2\kern-2pt E_6(q_0)$ &   $12, 18$    & $27$    &\\
$E_7$, $E_8$& $E_7(q_0), E_8(q_0)$ &   $18, 30$    & $56, 248$    &\\
\bottomrule
\end{tabular}
\label{t:c9-samechar2}
\end{table}

Since $V$ is an absolutely irreducible $\F_{p^a}G$-module which is realised over no proper subfield of $\F_{p^a}$, we use \cite[Proposition 5.4.6]{KL} (based on Steinberg's twisted tensor product theorem) to obtain a lower bound for $d=\dim(V)$. One of the following holds: 
\begin{enumerate}[(i)]
    \item For  $S$  untwisted,   \cite[Proposition 5.4.6(i)]{KL} implies that $b=af$ for some integer $f$, and $d=d_0^f$ where $d_0=\dim(M)$ for an irreducible $G$-module $M$ over an algebraically closed field of characteristic $p$; by \cite[Proposition 5.4.4]{KL} it follows that $d=d_0^f\geq R_p(S)^f$. 
   
 \item Next, for $S$ of type ${}^2\kern -2pt A, {}^2\kern -2pt D$ or ${}^2\kern -1pt E_6$,   \cite[Proposition 5.4.6(ii)]{KL} implies that $d=d_0^f\geq R_p(S)^f$, where $d_0=\dim(M)$ for an irreducible $S$-module $M$ over an algebraically closed field of characteristic $p$; and either $M$ is invariant under the corresponding graph automorphism $\tau_0$ defining $S$, and $b, a, f, d_0$ are as in case (i), or $M$ is not invariant under $\tau_0$ and we have $2b=af$ with $a$ even, $f$ an odd integer.
   
 \item For $S={}^3\kern -2pt D_4(q_0)$, \cite[Remark 5.4.7(a)]{KL} implies that $\delta b=af$ and $d=d_0^f$, where either $V$ is invariant under the corresponding graph automorphism $\tau_0$ defining $S$, $\delta=1$, and $d_0\geq 24$ (see \cite[Proposition 5.4.8]{KL}); or $V$ is not invariant under $\tau_0$, $\delta=3$, $3$ divides $a$ but not $f$, and $d_0\geq R_p(S)=8$. 
   
    \item Finally for $S$ one of the groups ${}^2\kern-2pt B_2(2^b)=\Sz(2^b), {}^2\kern-2pt F_4(2^b)'$, or  ${}^2\kern-1pt G_2(3^b)= {\rm Ree}(3^b)$, it follows from \cite[Remark 5.4.7(b)]{KL} that $b=af$ for some integer $f$, and $d\geq R_p(S)^f$. 
\end{enumerate}
 Now we consider several cases for the parameters $a,b,f$.

\medskip\noindent
\emph{Case $1$.\quad $b=af$ for some integer $f$ and $d\geq R_p(S)^f$, and $S$ is not $\PSU_m(q_0)$ or ${}^3\kern -2pt D_4(q_0)$.}  

Note that these conditions hold for all groups in parts (i) and (iv), and sometimes for the relevant groups in part (ii).
As we noted above, if these conditions hold then  $R_p(S)\geq \ell$, and  $r$ divides $q_0^x-1=p^{bx}-1=(p^{a})^{fx}-1$, for some $x\leq \ell$. Since $r$ is a primitive prime divisor of $(p^a)^{d/2}-1$, it follows that $d/2$ divides $fx$. Thus we have the following inequalities:
 \begin{equation}\label{e:c9-samechar1}
 2\ell f\geq 2 x f \geq d\geq R_p(S)^f\geq \ell^f.    
 \end{equation}
 Hence $2f\geq \ell^{f-1}$, and since $\ell\geq2$, one of: 
 \begin{center}
      $f=1$,\  or  $\ell=2$ with $f\in\{2, 3, 4\}$,\ or $\ell\in \{3, 4\}$ with  $f=2$. 
 \end{center}
 Suppose first that $\ell= 2$ with $f\in\{2, 3, 4\}$. Then, by Table~\ref{t:c9-samechar2},  $S=\PSL_2(q_0)$ so $S$ arises in part (i), and hence $d=d_0^f$ with $d_0$ even.   Thus $4f\geq d_0^f\geq 2^f$ by \eqref{e:c9-samechar1}, and this implies that $d_0=2$ and $d=2^f$. Also, as $2\leq x\leq \ell$, we have $x=2$, and so $d/2=2^{f-1}$ divides $xf=2f$, whence $f\in\{2,4\}$.  If $f=2$, then $S=\PSL_2(q^2)$ and  $G < \GL_{4}(q)$ with $r$ a ppd of $q^2-1$. However the only such irreducible quasisimple subgroups of $\GL_4(q)$ either preserve an extension field structure, so are  not absolutely irreducible,  or are classical groups of type $O_4^-$, see \cite[Tables 8.8--8.15]{BHRD}, so give no examples.
  If $f=4$, then $S=\PSL_2(q^4)$ and $G<\GL_{16}(q)$ with $r$ a ppd of $(q^4)^2-1$. There is a twisted tensor product subgroup $G$ of this form \cite[Table 1B, line 1]{S99} (and it is contained in a subgroup $\SL_4(q^2)$), as listed in line 1 of Table~\ref{t:c9-samechar}.
 
Suppose next that $\ell\in \{3, 4\}$ with  $f=2$, so $q_0=q^2$. Then \eqref{e:c9-samechar1} gives $16\geq 4\ell\geq 4x \geq d\geq \ell^2\geq 9$. This  implies that $x=\ell$, so $d/2$ divides $fx=2\ell$, and hence $d=4\ell$.    If $\ell=3$ then, by Table~\ref{t:c9-samechar2}, $S=\PSL_3(q_0)$, so $S$ arises in part (i). However, if $S$ is untwisted, as in part (i), then $d=d_0^f=d_0^2$ and we must have $d_0=\ell=4$. It follows that $\ell$ must be $4$, and $d=R_p(S)^2=16$. Then, by  Table~\ref{t:c9-samechar2}, $S=\PSL_4(q^2)$ or $\PSp_4(q^2)$ (noting that $S\ne \Sz(2^b)$ since $b=2a$ is even).  We get the examples in lines $2$ or $5$ of Table~\ref{t:c9-samechar}: 
$S=\PSL_4(q^2)$ or $\PSp_4(q^2)$ with $G=\GL_{16}(q)$ and $r$ a ppd of $(q^2)^4-1$. These are twisted tensor product subgroups \cite[Table 1B]{S99} on line 1 or 5 of \cite[Table 1B]{S99}, respectively.

Finally assume that $f=1$ so $q_0=q$ and $d/2$ divides $x$. Suppose first that $x>d/2$. Then $x\geq d$, and then \eqref{e:c9-samechar1} implies that $\ell=x=d=R_p(S)$. By the assumptions of Theorem~\ref{t:stingray}, $G$ is not a classical group as defined in \eqref{d:class}, and so the simple group $S$ is an exceptional Lie type group. Then by Table~\ref{t:c9-samechar2}, as $\ell= R_p(S)$,   either $S={}^2\kern-2pt B_2(q)$ with $d=4$, or $S=G_2(q)$ with $p=2$ and $d=6$. However $|{}^2\kern-2pt B_2(q)|$ is not divisible by a primitive prime divisor of $q^2-1$ and the former case does not give an example. On the other hand if $S=G_2(q)$ then $q^3-1$ divides $|S|$ giving an example as in line 12 of Table~\ref{t:c9-samechar}, see \cite[Table 8.29]{BHRD}; we note that by  \cite[Table 8.29]{BHRD}, $G<\Sp_6(q)$ and so $G$ does not contain a $3$-stingray element since all elements of order $r$ act fixed point freely on $V$. 

Thus we may assume that $x=d/2$, and by \eqref{e:c9-samechar1}, $\ell\leq R_p(S)\leq d\leq 2\ell$. Then by Table~\ref{t:c9-samechar2}, for the groups in Case 1,  either $S$ is a classical group, or $S\in\{{}^2\kern-2pt B_2(q), G_2(q), {}^2\kern-1pt G_2(q)', {}^2\kern-2pt E_6(q)\}$.  
 If $S={}^2\kern-2pt B_2(q)$ then, using the entries in Table~\ref{t:c9-samechar2}, $4=\ell\geq x=d/2\geq \ell/2=2,$ and since $r$ is a primitive prime divisor of $q^{d/2}-1=q^x-1$ and $r$ divides $|S|$, it follows on examining $|S|$ that $d/2=4$ so $d=8$. Exactly analogous arguments show that,  if $S=G_2(q)$ or $S={}^2\kern-1pt G_2(q)={\rm Ree}(q)$, then  $d/2=x=6$ so $d=12$, and if $S= {}^2E_6(q)$, then  $d/2=12$ so $d=24$. However by \cite[Theorem 5.7.5]{BHRD}, a quasisimple subgroup $G$ of $\GL_d(q)$ with $(d,G/\Z(G))=(8,{}^2\kern-2pt B_2(q))$, $(12, G_2(q))$, or $(12,{}^2\kern-1pt G_2(q))$,   must be contained in an extension of $O_8^+(q)$ or $\Omega_8^+(q)$, and by  \cite[Table 8.50]{BHRD}, there is no such subgroup. In the last case where $(d,G/\Z(G))=(24,  {}^2E_6(q))$, it follows from \cite[Proposition 5.4.8]{KL} that the parameter $d$ in Case 1 for this group is at least $72$, which is a contradiction. 

    This leaves the cases where $S$ is a classical group, namely one of $\PSL_m(q)$, $\PSp_m(q)$, or $\POm_m^\epsilon(q)$ (where $\epsilon\in\{+, -, \circ\}$). Then $\ell=m$ with $m\geq 2, 4, 7$, respectively (see Table~\ref{t:c9-samechar2}), and  $d\leq 2\ell$. First, if $\ell=m=2$ then $S=\PSL_2(q)$ and $d=4$; by \cite[Proposition 5.3.6(i) and Table 8.13]{BHRD} there is a unique conjugacy class of quasisimple subgroups $G=\SL_2(q)$ lying in $\Sp_4(q)$, as in line 6 of Table~\ref{t:c9-samechar}. Thus we may assume that $\ell=m\geq3$, and then we have $d\leq 2\ell=2m\leq \frac{1}{2}m(m+1), \frac{1}{2}m^2$, or $\frac{1}{2}m^2-1$, respectively, for the various families of classical groups. Thus the conditions of \cite[Proposition 5.4.11]{KL} hold: let $V_0=\F_q^m$ be the natural module for $S$. By assumption $V_0$ is not the natural module for $\GL_d(q)$, and in the light of the discussion in Remark~\ref{r:O8spin}, we have $m\ne d$.
    Then by \cite[Proposition 5.4.11]{KL}, $V, d, m$ are given by \cite[Table 5.4.A]{KL}. We examine each line and use the restrictions: $m\geq 3,4,7,8$ for $S=\PSL_m(q)$, $\PSp_m(q)$, $\POm_m^\circ(q)$, or $\POm_m^\pm(q)$ respectively; $d$ is  even and $4\leq d\leq 2m$. This produces the examples in lines 7--11 of Table~\ref{t:c9-samechar}, given the following comments. First,  for $S=\Sp_4(q)$ or $\POm_8^\pm(q)$, the $4$-dimensional irreducible section of $\Lambda^2 V_0$, or of the $8$-dimensional spin module, respectively, is the natural $S$-module, see \cite[Theorems 5.5.4 and 5.6.1, see Tables 8.8 and 8.44]{BHRD}, so do not give an example. Next, for $S=\PSL_4(q)$, the module $\Lambda^2 V_0$ is the natural $6$-dimensional module for $\Omega_6^+(q)$, by \cite[Proposition 5.4.2]{BHRD}, so does not give an example. Finally, for the $16$-dimensional spin modules for $S=\POm_{10}^\pm(q)$, we obtain the example in line $11$ for $P\Omega_{10}^+(q)$, but the spin representation  embeds $\Omega_{10}^-(q)$ into $\GL_{16}(q^2)$ and is not realisable over $\F_q$, see \cite[Proposition 5.4.9(iii)]{KL}.  This completes  Case 1.

\medskip\noindent
\emph{Case $2$.\quad $S={}^3\kern -2pt D_4(q_0)$.}

Here the conditions of part (iii) hold. Thus, for some integer $f$, we have $\delta b=af$, where either $\delta=1$ and $d=d_0^f\geq 24^f$, or $\delta = 3$ with $3\mid a$ but $3\nmid f$, and $d=d_0^f\geq 8^f$. By Table~\ref{t:c9-samechar2}, $\ell=12$. Now $r$ divides $q_0^x-1=p^{bx}-1=(p^{a/\delta})^{fx}-1$, which divides $(p^a)^{fx}-1$, for some $x\leq \ell=12$, and since $r$ is a primitive prime divisor of $(p^a)^{d/2}-1$, it follows that $d/2$ divides $fx$. Thus we have the following inequalities:
 \begin{equation*}
 24 f\geq 2 x f \geq d=d_0^f, \ \text{with $d_0\geq 24, 8$ for $\delta=1, 3$, respectively.}    
 \end{equation*}
 In particular $24f\geq 2xf\geq 8^f$, so $f=1$ and hence $q=q_0^\delta$, which is $q_0$ or $q_0^3$ for $\delta=1, 3$, respectively. Examining $|S|$ we see that $d/2$ divides $x$, and $x$ divides $12$. Thus either $(\delta,x,d)=(1,12,24)$, or $\delta=3$ and $x\geq4$.  Suppose first that $x=4$. Then $d=8$ so $\delta=3$ and  $r$ is a primitive prime divisor of $(q_0^{3})^4-1$. Since $r$ divides $|S|$ it follows that $r$ divides $q_0^8+q_0^4+1$, which implies that $r\nmid (q_0^4-1)$ (since $r\geq d/2+1=5> 3$), contradicting the fact that $x=4$.
Thus $x>4$, and as $x$ divides $12$, we have  $x=6y$ for some $y\leq 2$. 
Also $r$ divides $p^{bx}-1=p^{(a/\delta)6y}-1 = q^{y(6/\delta)}-1$, so $d/2$ divides $y(6/\delta)=x/\delta$. If $\delta=3$  then $4\leq d/2\leq x/\delta \leq 12/\delta=4$, so $d=8, x=12$, and we have $S={}^3\kern-2pt D_4(q^{1/3})$ which by \cite[Proposition 1.4.1]{K87} is  the centraliser in $\POm_8^+(q)$ of a triality automorphism and hence $G<\GL_8(q)$, see \cite[Table 8.50]{BHRD}, and line 13 of Table~\ref{t:c9-samechar} holds. On the other hand if $\delta=1$, then  $q_0=q$ (since $\delta b=af$) and $(\delta,x,d)=(1,12,24)$. By \cite[Proposition 1.4.1]{K87} and \cite[Table 4.3A and Proposition 4.3.14]{KL}, we have $S={}^3\kern -2pt D_4(q)< \POm_8^+(q^3) < \POm_{24}^+(q)$. However, by a theorem of Steinberg (see \cite[Theorem 5.4.1]{KL}) each absolutely irreducible representation of ${}^3\kern -2pt D_4(q)$ extends to a representation of the algebraic group of type $D_4$, and by \cite[Proposition 5.4.11]{KL}, the smallest dimension (greater than $8$) of such a representation is $28$ or $26$ for $q$ odd or even, respectively. Thus $S$ is not absolutely irreducible, and we have a contradiction.

\medskip
This leaves the groups in part (ii). First we treat the unitary groups.

\medskip\noindent
\emph{Case $3$.\quad   $S={}^2\kern -2pt A_{m-1}(q_0)=\PSU_m(q_0)$. }\quad

Now, we have $\delta b=af$, for some integer $f$ where either $\delta=1$, or $\delta = 2$ with $a$ even and $f$ odd. In either case $d=d_0^f\geq R_p(S)^f$, and by Table~\ref{t:c9-samechar2}, $R_p(S)=m\geq3$ and $\ell=2m$. Now $r$ divides $q_0^x-1=p^{bx}-1=(p^{a/\delta})^{fx}-1$, which divides $(p^a)^{fx}-1$, where $x\leq \ell=2m$, and since $r$ is a primitive prime divisor of $(p^a)^{d/2}-1$, it follows that $d/2$ divides $fx$. Thus we have the following inequalities:
 \begin{equation}\label{e:c9-samechar2}
 4 mf\geq 2 x f \geq d=d_0^f\geq R_p(S)^f= m^f.    
 \end{equation}
 Hence $4f\geq m^{f-1}$, so one of $f=1$, or $f=2$ with $3\leq m\leq 8$, or $f=m=3$. The last case does not hold since it implies $d=3^3$ while $d$ is even. 
 Suppose that $f=2$, so $\delta=1$ (as $f$ is even), $q_0=q^2$, and $d=d_0^2$ with $d_0=\dim(M)$ for an irreducible $S$-module $M$ (over the algebraic closure of $\F_p$) which is invariant under the corresponding graph automorphism $\tau_0$. The latter condition implies by \cite[Proposition 5.4.8]{KL} that $d_0\geq \max\{6, m\}$. Since $d$ is even, also $d_0$ is even, say $d_0=2d_1$, so $d/2=2d_1^2$ is even. Therefore $r$ divides $(q^2)^{d/4}-1=q_0^{d/4}-1$, and so $x\leq d/4$, and it follows from \eqref{e:c9-samechar2} that $x=d/4$.  
 Now
 \[
|S|= |\PSU_m(q_0)|=\frac{1}{\gcd(m,q_0+1)}q_0^{m(m-1)/2}\prod_{i=2}^m(q_0^i- (-1)^i).
 \]
 If $d/4=d_1^2$ is even then $q_0^{d/4}-1$ is one of the factors in the product for $|S|$ and so $d\leq 4m$; while if $d/4$ is odd then the first factor in the product divisible by $r$ is $q_0^{d/2}-1$ and so $d\leq 2m$. In either case $4m\geq d=4d_1^2=d_0^2\geq \max\{6, m\}^2\geq m^2$, and this implies both that  $4\geq m$ and $4m\geq 6^2$, which is a contradiction.

Thus  $f=1$, and we have $4m\geq 2x\geq d\geq m$, and $\delta=1$ or $2$. Suppose first that $\delta=2$, so $q_0^2=q$. 

We note that, by Remark~\ref{r:U}, if $(d,q)= (m,q_0^2)$, then $G$ is a classical subgroup of $\GL_d(q)$, and we are assuming that this is not the case. Hence $d\ne m$,
 
and we have $4m\geq 2x\geq d> m$. 
If $x$ is odd then the first factor in the product for $|S|$ which is divisible by $r$ is $q_0^{2x}-1$ so $2x\leq m$, which contradicts the inequality $2x\geq d>m$. 
Thus $x$ is even, so $r$ divides $q_0^x-1=q^{x/2}-1$ and hence $d/2$ divides $x/2$. This implies that $x= dy$ for some positive integer $y$, and we have $4m\geq 2x=2dy> 2ym$, and hence $y=1$ and $x=d$, so $2m\geq x=d>m$. 
Now  $q_0^x-1$ (if $x/2$ is even), or $q_0^{x/2}+1$ (if $x/2$ is odd), is the first factor in the product for $|S|$ that is divisible by $r$; and since $x>m$ we must have $x/2=d/2$ odd. Further, since $m<d=x\leq 2m \leq \frac{1}{2}m(m+1)$, it follows from \cite[Proposition 5.4.11]{KL} that $d=\frac{1}{2}m(m+1), \frac{1}{2}m(m-1)$, or $(d,m)=(20,6)$. Since $d/2$ is odd and $d\leq 2m$ it follows that $(d,m)=(6,3)$ or $(10,5)$, as in line 4 of Table~\ref{t:c9-samechar}, or $(d,m)=(6,4)$. However, by \cite[Proposition 5.4.2]{BHRD}, in the latter case the $S$-module is the natural module for $\Omega_6^-(q)$, and so does not give an example.

Now suppose that $\delta=1$, so $q_0=q$ and also $x=d/2$, by the definition of $x$. If $d/2$ is odd then the first factor in the product for $|S|$ which is divisible by $r$ is $q^{2x}-1 = q^d-1$ so $d\leq m$. This however implies that $d=m$, but the group $S=\PSU_d(q)$ is not involved in $\GL_q(q)$. Thus $d/2$ is even, and the first factor in the product for $|S|$ which is divisible by $r$ is therefore $q^{d/2}-1$ if $d/4$ is even (so $d/2\leq m$), or $q^{d/4}+1$ if $d/4$ is odd (so $d/4\leq m$). In either case $d\leq 4m$, and since $m\geq3$ this implies that either $d\leq \frac{1}{2}m(m+1)$, or $(d,m)=(20,5), (12,4)$ or $(12,3)$. There are no possibilities for the cases $(d,m)=(12,4)$ or $(12,3)$ by \cite[Proposition 5.4.1]{BHRD}. Moreover, if $d\leq \frac{1}{2}m(m+1)$, then by \cite[Proposition 5.4.11]{KL}, $d= \frac{1}{2}m(m-1)$, or $d=\frac{1}{2}m(m+1)$, or $(d,m)=(20,6)$. Recall that $d/2$ is even and that, if $d/4$ is even, then $d\leq 2m$. None of the possibilities for $d, m$ satisfies $d\leq 2m$. Hence $d/4$ is odd and at most $m$. The possibilities are $(d,m)=(20,6)$, $(28,7)$, $(28,m)$, or $(36,9)$,  as in line 3 of Table~\ref{t:c9-samechar},  or $(d,m)=(20,5)$. However there is no absolutely irreducible representation for $\PSU_5(q)$ in dimension $20$ and characteristic $p$, see \cite[Appendix A.8]{Lu}.

\medskip
Finally we consider the remaining twisted groups in part (ii), for the situation not covered in Case 1. 

\medskip\noindent
\emph{Case $4$.\quad  $S={}^2\kern-2pt D_{m/2}(q_0)=\POm^-_m(q_0)$ ($m$ even, $m\geq8$), or $S={}^2\kern-2pt E_{6}(q_0)$, with $2b=af$, where $a$ is even and $f$ is odd. }\quad
 
By the description in part (ii) we have $d=d_0^f\geq R_p(S)^f$, where by Table~\ref{t:c9-samechar2}, $R_p(S)=m\geq8$ and $\ell=m$ if $S=\POm^-_m(q_0)$, and 
$R_p(S)=27$ and $\ell=12$ if $S={}^2E_{6}(q_0)$. 
Now $r$ divides $q_0^x-1=p^{bx}-1=(p^{a/2})^{fx}-1$, which divides $(p^a)^{fx}-1$, for some (least possible) $x\leq \ell$, and since $r$ is a primitive prime divisor of $(p^a)^{d/2}-1$, it follows that $d/2$ divides $fx$. Thus we have the following inequalities:
 \begin{equation*}
 2\ell f\geq 2 x f \geq d=d_0^f\geq R_p(S)^f.    
 \end{equation*}
If $S={}^2E_{6}(q_0)$ these inequalities imply that $24f=2\ell f\geq R_p(S)^f=27^f$ which is a contradiction. Thus $S=\POm^-_m(q_0)$, and here the inequalities imply that $2f\geq m^{f-1}$. Hence $f=1$ (since $m\geq 8$), so $q_0^2=q$ and  $d/2$ divides $x$; in particular $x\geq d/2\geq m/2$. Note that 
\[
|S|= |\POm^-_m(q_0)|=\frac{1}{\gcd(4,q_0^{m/2+1})}q_0^{m(m-2)/4} (q_0^{m/2}+1)\prod_{i=1}^{m/2-1}(q_0^{2i}- 1).
 \]
Suppose first that $x$ is odd. Then the first factor in the product for $|S|$ that is divisible by $r$ is $q_0^{2x}-1$ with $2x\leq m-2$, (since if this first factor were $q_0^{m/2}+1$ then $x=m$ would be even). 
Since $2x\geq d\geq m$, we have a contradiction. Thus $x$ is even, and so  the first factor in the product for $|S|$ that is divisible by $r$ is either $q_0^{x}-1$ with  $x/2\leq m/2-1$, or $q_0^{m/2}+1$ with $x=m$. In particular $x\leq m$. Also $r$ divides $q_0^{x}-1=q^{x/2}-1$, so $d/2$ divides $x/2$, and hence  $m\geq x\geq d\geq m$. Thus  $d=m=x$ and the group $S=\POm^-_d(q^{1/2})$. Since by assumption $G$ is not realisable over $\F_{q_0}$, it follows from \cite[Proposition 5.4.11]{KL} that $d=m=8$ with $G$ acting on an spin module over $\F_{q_0}$. By \cite[Table 8.50]{BHRD} there is indeed a quasisimple subgroup $G=d.S < \Omega_8^+(q)$, as in line $7'$ of Table~\ref{t:c9-samechar}. This completes the analysis, and hence the entries in Table~\ref{t:c9-samechar} are correct. 

\subsection{Determining the list of quasisimple groups containing a \texorpdfstring{$(d/2)$}{}-ppd stingray element for Proposition~\ref{p:same}}\label{s:same-stingray}

We assume now that $X=\GL_d(q)$, where $d\geq 4$, $d$ is even, and $q=p^a$ for a prime $p$ and positive integer $a$,  let $Z$ be the subgroup of scalars in $X$, and $V=\F_q^d$ be the underlying natural $X$-module. Suppose also that $G<X$ such that $S\coloneq GZ/Z, d, q$ are as in one of the lines of Table~\ref{t:c9-samechar}, and that $G$ contains an element $g$ of prime order $r$, where $r$ is a primitive prime divisor of $q^{d/2}-1$, such that $g$ is a $(d/2)$-stingray element on $V$. Using the notation in  \cite{Z23}, we let $m_g(V)$ denote the maximum eigenvalue multiplicity for $g$ (over an algebraic closure of $\F_q$) and we note that $m_g(V)=d/2$, since $g$ is a $(d/2)$-stingray element.  

We consider the lines of Table~\ref{t:c9-samechar}, according to the groups $S$. First we deal with $\PSL_2(q')$. We note that groups with socle $\PSL_2(q^2)$ occur in $\PSL_4(q)$ as classical orthogonal groups of minus type, but these groups are excluded from our analysis here by our assumption in Theorem~\ref{t:stingray} that $G$ does not contain a classical group. We make some comments about the representations of $\SL_2(q)$ which will be useful in the proof.

\begin{remark}\label{r:sl2}
     Let $K$ be an algebraically closed field containing $\mathbb{F}_q$, where $q=p^a$, and suppose that $\rho:\SL_2(q)\to\GL_4(K)$ is absolutely irreducible and is a Frobenius twist $\rho=\phi\otimes \psi$ of irreducible representations $\phi,\psi:\SL_2(q)\to\GL_2(K)$.
     The theorem of Brauer and Nesbitt  (see \cite[Theorem 5.3.2]{BHRD}) describes the absolutely irreducible representations of $\SL_2(q)$ as modules $M(n)$ for some (unique) $n$ with $0\leq n\leq p^a-1$, where, if we write $n=a_0+a_1p+\dots+a_{a-1}p^{a-1}$  with $0\leq a_j\leq p-1$, then $\dim M(n)=\prod_{i=0}^{a-1}(a_i+1)$. For the precise definition of $M(n)$ see \cite[\S 5.3]{BHRD}. The module $M_\phi$ corresponding to $\phi$ has dimension $2$, and this means that $M_\phi=M(p^s)$ for some $s<a$. Note that $M(0)$ is the natural module, and $M(p^s) = M(0)^{\sigma^{s-1}}$ in the notation of \cite{BHRD}, where $\sigma$ is the automorphism $x\to x^p$ of $\F_{p^a}$.

    Thus $M_\phi=M(p^s)$ and $M_\psi=M(p^t)$ for some $s, t<a$, and the module for $\rho$ is $M_\rho =M(p^s)\otimes M(p^t)$. If $s=t$ then  $ \rho$ is reducible as $M_\rho$ is then not equal $M(i)$ for any $i\in\{0, \dots, p^a-1\}$. (See \cite[1.6]{Se} for a general result on reducibility of tensor products.) Thus $s\ne t$ and then $M_\rho= M(p^s+p^t)$, which  is one  of these absolutely irreducible modules in \cite[Theorem 5.3.2]{BHRD}, and in this case $\max\{s,t\}\geq 1$ so $a>1$.\qed
\end{remark}

\begin{lemma}\label{l:psl2}
\begin{enumerate}
    \item[(a)]      For $S=\PSL_2(q^4)$ in line~$1$ of Table~$\ref{t:c9-samechar}$, $G$ does not contain an $8$-stingray element $g$ of ppd prime order $r \mid (q^4+1)$.
    \item[(b)] For $S=\PSL_2(q)$ in line~$6$ of Table~$\ref{t:c9-samechar}$, either
    \begin{enumerate}
        \item[(i)] $G=\SL_2(q)<\Sp_4(q)$, and $G$ contains a $2$-stingray element $g$ of odd prime order $r$ if and only if $q\equiv 2\pmod{3}$ and $r=|g|=3$; or
        \item[(ii)] $G=\PSL_2(q)$, and $G$ contains a $2$-stingray element $g$ of odd prime order $r$ if and only if $q=p^{a'c}$ with $c$ odd, and $c>1$, $p^{a'}+1$  is not a power of $2$,  and $r\mid (p^{a'}+1)$.
    \end{enumerate}
\end{enumerate}
\end{lemma}

\begin{proof}
    Let $K$ be an algebraically closed field containing $\mathbb{F}_q$, and let $H=\SL_2(q')$, where $q'=q^4$ for line 1 and $q'=q$ for line 6 of Table~\ref{t:c9-samechar}. Let $\rho:H\to\GL_d(q)$ denote the representation with $\rho(H)=G$ and let $h\in H$ such that $\rho(h)=g$. Note that the prime $r$ divides $q^4+1$ or $q+1$ in lines 1, 6, respectively, and in both cases $h$ is a regular semisimple element of $H$.
    
By the general representation theory of algebraic groups, every absolutely  irreducible representation of $H$  extends to an  irreducible representation of a simple algebraic group $\mathbf{Y}$, see \cite[Theorem 43]{St}. Since $H= \SL_2(q')$ we have $\mathbf{Y}\cong \SL_2(K )$, and it follows from \cite[Corollary 1.3]{Z23}, and $m_g(V)=d/2$, that $d=4$. Hence no $8$-stingray element exists in line 1, proving part (a), and we are therefore in line 6 with $e=d/2=2$ and $q'=q$.  Thus we can use \cite[Lemma 5.7]{Z23} with $\dim V=4$, and
hence case (2) or (3) of \cite[Lemma 5.7]{Z23} holds. 

Consider first case (2)  of  \cite[Lemma 5.7]{Z23}. We see in the proof of \cite[Lemma 5.7]{Z23} that   case (2) arises from an  irreducible representation of $\mathbf{Y}$ with highest weight $3\om_1$. This representation can be realized in the space of homogeneous polynomials of degree~$3$ in two variables, see for instance   \cite[Theorem 10.1.8, p.\;114]{Bn}. This representation 
is not tensor-decomposable for $\mathbf{Y}$ \cite[(1.6)]{Se}, and hence also not for $H$. (This is obvious for $q$ odd: as $\Z(H)$ is non-trivial in this  representation, while $\Z(H)$ is trivial in every tensor-decomposable  representation of degree $4$.)  In fact, $G=\rho(H)\subset \Sp_4(q)$, see \cite[Proposition 5.3.6(i)]{BHRD} or \cite[Lemma 79]{St}. We see as follows that the elements obtained in  \cite[Lemma 5.7(2)]{Z23} really do lie in a conjugate of $G$. We must have, $h^3\in \Z(H)$ by \cite[Lemma 5.7]{Z23}, whence $|\rho(h)|=|g|=3$. The Sylow $3$-subgroups of $H$ are cyclic; and this implies that elements of order $3$ in $H$ are conjugate to those in $\SL_2(p)$. Therefore, $h$ can be taken to lie in $H $. In addition, $3\mid(q+1)$ as otherwise $\rho(h)=g$ is diagonalisable, hence not a $2$-stingray element in $\GL_4(q)$. 

Now consider case (3) of  \cite[Lemma 5.7]{Z23}. As in the proof of \cite[Lemma 5.7]{Z23}, the representation over $K$ is a tensor product $\rho=\rho_1\otimes \rho_2$ (see Remark~\ref{r:sl2}), where $\rho_1,\rho_2$ are inequivalent irreducible representations of $H$ of degree $2$,  $a>1$, and $\rho(H)=\PSL_2(q)$. Further, over $K$, semisimple elements of $\rho(H)$ are of the form $\diag(t,t^{-1})$.  
We can write $\rho_1(h)=\diag(t_1,t_1^{-1})$ and $\rho_2(h)=\diag(t_2,t_2^{-1})$ for $t_1,t_2\in K$. Hence $\rho(h)$ has eigenvalue $1$ if and only if $t_2\in\{t_1,t_1^{-1}\}$, and then   $\rho(h)=\diag(1,1,t_1^2,t_1^{-2})$. In addition, the $2$-stingray element $g=\rho(h)$ is required to be  non-diagonalizable  in $\GL_4(q)$, so $t_1^2\notin \F_q$. As $r=|g|$ is odd, this means that $t_1\notin \F_q$. 

Now we also have $\rho_2=\rho_1^\gamma$ for some field automorphism $\gamma$, where $\gamma\ne 1$ since $\rho_1,\rho_2$ are inequivalent (see Remark~\ref{r:sl2}). Moreover we must have $\gamma^2\neq 1$, as otherwise $q$ is a square and $\rho_1\otimes \rho_2$ can be realized over  $\mathbb{F}_{\sqrt{q}}$, and we are assuming that this is not the case. Let $\F_{q_0}$ be the fixed field of $\gamma$. Then $q_0=p^{a'}$ and $q=p^a=p^{a'c}$, where $c=|\gamma|>2$.

The hypotheses of \cite[Lemma 5.1]{Z23} hold for the  $2$-stingray element $g=\rho(h)$ and the `highest weight $(1+p^{a'})\omega_1$', and this result implies that $h$ is conjugate in $\SL_2(K)$ to an element of $\SL_2(q_0)$. 

We now wish to apply Lemma 5.1 of \cite{Z23}, and in order to do so we need to interpret the notation of that lemma in terms of our notation. For this 
let $f:K\rightarrow K$ be the mapping defined by $f(x)=x^p$ for $x\in K$. Then $f$ stabilises every finite subfield $\mathbb{F}_{p^e}$ of $K$ and induces on it a Galois automorphism of $\mathbb{F}_{p^e}/\mathbb{F}_p$
of order $e$. An advantage of using $f$ rather than Galois automorphisms is that every automorphism of $\mathbb{F}_q$ automatically extends to $K$ and 
to every finite subfield $F$ with $\mathbb{F}_q\subset F\subset K$. Note that $f$ induces an automorphism of $\GL_2(K)$
and is called the standard Frobenius morphism $\GL_2(K)\rightarrow \GL_2(K)$. Moreover,
if $y\in \GL_2(K)$ is diagonalizable, say  $MyM^{-1}=\diag(x_1,x_2)$ with $x_1,x_2\in K$ for some $M\in \GL_2(K)$, then $f(y)=\diag(f(x_1),f(x_2))$. 

Next we mimic the proof of \cite[Lemma 5.1]{Z23} for the reader's convenience. The automorphism $\gamma$ is equal to $f^{a'}$, and we recall that the subfield of $\F_q$ fixed by $\gamma$ is $\F_{p^{a'}}=\F_{q_0}$. 
The element $h\in H$ such that $g=\rho(h)$ is conjugate to $\diag(t,t^{-1})$ in $\SL_2(K)$, for some $t\in K$. We may assume that $\rho_1$
is the natural representation, so that $\rho_1(h)=h$, and then $\rho_2(h)=\rho_1(\gamma(h))$ is conjugate to $\diag(t^{p^{a'}}, t^{-p^{a'}})$. Therefore $\rho(h)=\diag(t^{p^{a'}+1},t^{p^{a'}-1},t^{-p^{a'}+1},t^{-p^{a'}-1})$. As $\rho(h)$ is a $2$-stingray element, two of these entries equal $1$, so either   $t^{p^{a'}+1}=1=t^{-p^{a'}-1}$ or $t^{p^{a'}-1}=1=t^{-p^{a'}+1}$.  
In the latter case $t\in   \mathbb{F}_{q_0}$, and hence $h$ is  conjugate to an element of $\SL_2(q_0)$. 
However, this implies that $\rho(h)$ is diagonalisable in $\GL_4(q)$, so $\rho(h)$ is not a stingray element. Hence the former case holds and here $|t|$ divides $q_0+1$, and $\rho(h)=\diag(1,t^{-2},t^{2},1)$. As $\SL_2(q_0)$ contains an element of order
$q_0+1$, it easily follows that $h$ is  conjugate to an element of $\SL_2(q_0)$.
We note the following for $q=q_0^c$: if $c$ is odd then $(q_0+1,q-1)=(q_0,2)$ which is not divisible by $|t|$ and hence $\rho(h)$ is   not diagonalizable. On the other hand, if $c$ is even then $|t|$ divides $q_0+1$, which divides $q_0^2-1$ and hence divides $q-1$, so   $\rho(h)$ is   diagonalisable in $\GL_4(q)$. In summary, $\rho(h)$ is   diagonalisable in $\GL_4(q)$ if and only if $\mathbb{F}_{q_0^2}\subseteq \mathbb{F} _q$.

Thus $\rho(H)$ contains $2$-stingray elements if and only if $a=a'c$ with $c$ odd and as mentioned above, $c>1$ (since $\rho_1, \rho_2$ are inequivalent); and in this case $\rho(H)$ does contain $2$-stingray elements of any order $r$ dividing $q_0+1=p^{a'}+1$ such that $r>2$. In particular in this case $\rho(H)$ contains a $2$-stingray element of order $r$ for each odd prime dividing $p^{a'}+1$ (if such exists), and each prime $r$ is a primitive prime divisor of $q^2-1$ since $(q_0+1,q-1)\leq 2$.   
Thus part (b) is proved.
\end{proof}

Now we consider the groups of Lie type in Table~\ref{t:c9-samechar} other than  $\PSL_2(q')$ for $p$-powers $q'$. 
Let $H$ be a  quasi-simple group of Lie type in defining characteristic $p$ such that $\rho:H\to\GL_d(q)$ has image $\rho(H)=G$. We may assume that $H$ is universal subject to the condition  $(|\Z(H)|,p)=1$.  
Let $h\in H$ be such that $\rho(h)$ is the $(d/2)$-stingray element $g$ of order $r$. Recall that the absolutely irreducible representation $\rho$ of $H$ extends to a  representation of a simple universal (or simply connected)  algebraic group $\mathbf{H}$.

Now a $(d/2)$-stingray element $g=\rho(h)\in\GL(V)$ is almost cyclic (see Section~\ref{s:strategy} and Remark~\ref{r:e-ppd}), and also $g$ is semisimple with  $m_g(V)=d/2$ (see the beginning of Subsection~\ref{s:same-stingray}), and so we can use the analysis in \cite{Z23}. This distinguishes between the cases with $h$ regular and non-regular (as an element of $H$). Note that if $h$ is regular semisimple then $(p,|C_G(g)|)=1$. 

If $h$ is regular, and hence regular semisimple, then by \cite[Theorem 1.4]{Z23}, the condition that $g$ is a $(d/2)$-stingray element with $m_g(V)=d/2$  implies that $H=\SL_2(q')$ for some $p$-power $q'$, which is not the case. Hence $h$ is not regular. Then, as $h$ is almost cyclic, it follows from \cite[Theorem 5.11]{Z23} that $H$ must lie in the set 

\begin{align}
    \mathcal{S}\coloneq \{\SL_{n+1}(q'), \SU_{n+1}(q'),\;&\Sp_{2n}(q'),  n>1\}\;\cup\notag \\ &\{\Spin_{2n+1}(q'), q'{\rm\ odd\ }, \Spin_{2n}^\pm(q'), n\geq 3\},\label{e:S}
\end{align}
for some $p$-power $q'$, and that the degree $d$ is equal to $n+1,n+1,2n$, $2n+1$, $2n$, respectively. Observe that
$\SL_4(q')\cong \Spin^+_6(q')$ and $\SU_4(q')\cong \Spin_6^-(q')$. We compare this information with the list from Table~\ref{t:c9-samechar}.

\begin{lemma}\label{l:same-notsl}
If $S\not\cong\PSL_2(q')$ for any $p$-power $q'$, and $S$ occurs in a line of Table~$\ref{t:c9-samechar}$, then the group $G$ in that line  does not contain a $(d/2)$-ppd stingray element.
\end{lemma}

\begin{proof}
    The groups  $S=G_2(q)$ and ${}^3\kern -2pt D_4(q^{1/3})$ do not arise in the set $\mathcal{S}$ in~\eqref{e:S}, and hence lines 12 and 13 of Table~\ref{t:c9-samechar} are ruled out.
    The groups $S=\PSL_m(q')$ or $\PSU_m(q')$ for some $m\geq 3$,  arise in lines 2, 3, 4 and 7 of Table~\ref{t:c9-samechar} and, by our comments above, for these groups either the degree $d$ should be $m$, or $(d,m)=(6,4)$. As this is not the case, these lines are ruled out. 
    A similar check rules out lines 5 and 10 for $S=\PSp_m(q')$, line 8 for $S=\POm_m^\circ(q)$, and line 11 for $S=\POm_{10}^+(q)$. 
    
    The only case remaining is line 9 with $S=\POm_8^-(q^{1/2})$. Let $q_0=q^{1/2}$.  Comparing with the set $\mathcal{S}$ we see that in this case there is a faithful irreducible representation $\rho:H\to\GL_8(q)$ with $G=\rho(H)\cong H=\Spin_8^-(q_0)$ acting on $V=\F_q^8$, and as noted in Table~\ref{t:c9-samechar}, $G\leq \Omega_8^+(q)<\GL_8(q)$. We emphasise that the representation $\rho$ is not realisable over a proper subfield, as is noted in the proof of \cite[Proposition 5.4.9(iii)]{KL} which establishes the embedding of $\rho(H)$ in $\Omega_8^+(q)$, and is confirmed by its inclusion in \cite[Table 8.50, on p. 403]{BHRD} (as  $H\in \mathcal{S}$). The  element $g=\rho(h)$ has order $r$, a primitive prime divisor of $q^4-1$, so $r$ divides $q^2+1 = q_0^4+1$ and $r\geq d/2+1=5$. It follows that, in the natural representation $\phi:H\to \GL_8(q_0)$ on  $V_0=\F_{q_0}^8$, the element $\phi(h)$ is irreducible (noting that $r$ does not divide $q_0^i-1$ for any $i<8$ since $\gcd(q_0^4+1, q_0^i-1)=\gcd(q_0,2)$).    Since $\phi(h)$ is irreducible on $V_0$ it follows that $h$ is a regular semisimple element of $H$ (as otherwise $C_H(h)$ would contain a unipotent element). Therefore, if $g=\rho(h)$ is a $4$-stringray element then we have a contradiction to  \cite[Theorem 1.4]{Z23}. Thus $G$ contains no $4$-ppd stingray elements. This completes the proof of the lemma.
    \end{proof}

\medskip\noindent
\emph{Proof of Proposition~\ref{p:same}:}\quad The first assertion of Proposition~\ref{p:same}, justifying the entries in the first three columns of Table~\ref{t:c9-samechar}, is proved in Subsection~\ref{s:same-ppd}. The final assertions of Proposition~\ref{p:same} concerning $(d/2)$-ppd stingray elements of prime order follow from Lemmas~\ref{l:psl2} and~\ref{l:same-notsl}.

\section{Acknowledgements}

We remember  our colleague Richard  Parker, to whom this paper is dedicated, for  his infectious
enthusiasm  for  finite  simple  groups. We are grateful to him for  providing  essential
computational tools  to work with  them. As  a key contributor  to the
Atlas project,  Richard   played a significant  role in  shaping it
into the invaluable resource for group theory it is today.  This paper
 greatly  benefited from  the information  on finite  simple groups
available  in  the Atlas  and  Modular  Atlas, accessible  through  the
AtlasRep package in GAP, which Richard co-authored.

We thank Thomas Breuer, a  co-author of the AtlasRep package, for
his valuable contributions  to its development. His advice
has been instrumental  in ensuring the accuracy of the  tables in this
paper.   We  are grateful  to  Gerhard  Hi{\ss}  and Gunter  Malle  for
providing current versions of their  tables of  absolutely irreducible  representations of
Lie-type groups in non-defining characteristic, which formed the basis
of our computational tests.

We   thank Tim Burness  and Frank L{\"u}beck for  their insightful
discussions on some  of the groups we encountered,  and Martin Liebeck
for his helpful  advice, especially for directing us  to the work
of M.\ Schaffer~\cite{S99}.   We thank Colva Roney-Dougal for her advice  and
confirming that certain small
degree quasisimple $\cC_9$-subgroups are not maximal.

The second author acknowledges support by the German Research Foundation (DFG) -- Project-ID 286237555 -- within the SFB-TRR 195 ``Symbolic Tools in Mathematics
and their Applications''.
The third and fourth authors would like to thank the Isaac Newton Institute for Mathematical Sciences, Cambridge, for support and hospitality during the programme Groups, representations and applications: new perspectives, where work on this paper was undertaken. This work was supported by EPSRC grant EP/R014604/1. The work forms part of the Australian Research Council Discovery Project DP190100450 of the first three authors.

\end{document}